\documentclass[a4paper,reqno]{amsart}

\usepackage{amsthm}
\usepackage{amsmath}
\usepackage{amssymb}
\usepackage{mathrsfs}
\usepackage{centernot}   
\usepackage{graphicx}        % pacchetto grafico standard LaTeX 
\usepackage{multicol}        % usato per l'indice a due colonne
\usepackage[all,ps]{xy}
\usepackage{layout}

\newcommand\aff{{\mathbb A}}

\newcommand{\CC}{{\mathbb C}}
\newcommand{\cA}{{\mathscr A}}

\newcommand{\cB}{{\mathscr B}}

\newcommand{\cD}{{\mathscr D}}

\newcommand{\cH}{{\mathscr H}}

\newcommand{\cM}{{\mathscr M}}

\newcommand{\cO}{{\mathscr O}}
\newcommand{\cP}{{\mathscr P}}

\newcommand{\cU}{{\mathscr U}}
\newcommand{\cV}{{\mathscr V}} 
\newcommand{\cW}{{\mathscr W}}
\newcommand{\cX}{{\mathscr X}} 
\newcommand{\cY}{{\mathscr Y}}
\newcommand{\cZ}{{\mathscr Z}}

\newcommand{\DD}{{\mathbb D}}
\newcommand{\dra}{\dashrightarrow}

\newcommand{\es}{\emptyset}

\newcommand{\gC}{\mathfrak{C}}

\newcommand{\gI}{\mathfrak{I}}

\newcommand{\gM}{\mathfrak{M}}
\newcommand{\gN}{\mathfrak{N}}

\newcommand{\gp}{\mathfrak{p}}

\newcommand{\gs}{\mathfrak{s}}

\newcommand{\gT}{\mathfrak{T}}

\newcommand{\hra}{\hookrightarrow}
\newcommand{\la}{\langle}

\newcommand{\lagr}{\mathbb{LG}(\bigwedge^ 3 V)}
\newcommand{\lagrdual}{\mathbb{LG}(\bigwedge^ 3 V^{\vee})}

\newcommand{\lagrhat}{\widehat{\mathbb{LG}}(\bigwedge^ 3 V)}

\newcommand{\lra}{\longrightarrow}

\newcommand{\ov}{\overline}
\newcommand{\PP}{{\mathbb P}}
\newcommand{\QQ}{{\mathbb Q}}
\newcommand{\ra}{\rangle}
\newcommand{\RR}{{\mathbb R}}

\newcommand{\XX}{{\mathbb X}}
\newcommand{\wh}{\widehat}
\newcommand{\wt}{\widetilde}
\newcommand{\ZZ}{{\mathbb Z}}
\newcommand{\Gr}{\mathrm{Gr}}
\newcommand{\mapor}[1]{{\stackrel{#1}{\longrightarrow}}}
\newcommand{\mapver}[1]{\Big\downarrow
\vcenter{\rlap{$\scriptstyle#1$}}}

\theoremstyle{plain}

\newtheorem{thm}{Theorem}[section]

\newtheorem{clm}[thm]{Claim}

\newtheorem{crl}[thm]{Corollary}

\newtheorem{lmm}[thm]{Lemma}
\newtheorem{prp}[thm]{Proposition}
\newtheorem{prp-dfn}[thm]{Proposition-Definition}

\theoremstyle{definition}

\newtheorem{dfn}[thm]{Definition}

\theoremstyle{remark}

\newtheorem{rmk}[thm]{Remark}

\DeclareMathOperator{\Ann}{Ann}

\DeclareMathOperator{\Card}{Card}

\DeclareMathOperator{\cod}{cod}

\DeclareMathOperator{\cork}{cork}
\DeclareMathOperator{\Def}{Def}
\DeclareMathOperator{\discr}{discr}
\DeclareMathOperator{\divisore}{div}

\DeclareMathOperator{\GL}{GL}
\DeclareMathOperator{\Hom}{Hom}
\DeclareMathOperator{\Id}{Id}
\DeclareMathOperator{\im}{Im}
\DeclareMathOperator{\Ind}{Ind}

\DeclareMathOperator{\mult}{mult}

\DeclareMathOperator{\PGL}{PGL}
\DeclareMathOperator{\Pic}{Pic}

\DeclareMathOperator{\rk}{rk}

\DeclareMathOperator{\sing}{sing}

\DeclareMathOperator{\Stab}{Stab}
\DeclareMathOperator{\supp}{supp}
\DeclareMathOperator{\Sym}{S}

\DeclareMathOperator{\Tr}{Tr}
\DeclareMathOperator{\vol}{vol}

\usepackage{ifthen}
\usepackage{rotating}
\usepackage{supertabular}
\newcommand{\cit}[1]{{\rm \textbf{#1}}}
\newcommand{\Ref}[2]{\cit{%
\ifthenelse{\equal{#1}{thm}}{Theorem}{}%
\ifthenelse{\equal{#1}{ass}}{Assumption}{}%
\ifthenelse{\equal{#1}{chp}}{Chapter}{}%
\ifthenelse{\equal{#1}{prp}}{Proposition}{}%
\ifthenelse{\equal{#1}{lmm}}{Lemma}{}%
\ifthenelse{\equal{#1}{cnj}}{Conjecture}{}%
\ifthenelse{\equal{#1}{crl}}{Corollary}{}%
\ifthenelse{\equal{#1}{dfn}}{Definition}{}%
\ifthenelse{\equal{#1}{expl}}{Example}{}%
\ifthenelse{\equal{#1}{hyp}}{Hypothesis}{}%
\ifthenelse{\equal{#1}{rmk}}{Remark}{}%
\ifthenelse{\equal{#1}{clm}}{Claim}{}%
\ifthenelse{\equal{#1}{exe}}{Exercise}{}%
\ifthenelse{\equal{#1}{qst}}{Question}{}%
\ifthenelse{\equal{#1}{sec}}{Section}{}%
\ifthenelse{\equal{#1}{subsec}}{Subsection}{}%
\ifthenelse{\equal{#1}{univ}}{Universal Property}{}%
\ifthenelse{\equal{#1}{trm}}{Terminology}{}%
\ifthenelse{\equal{#1}{tbl}}{Table}{}%
\  \ref{#1:#2}%
}}

\usepackage{multirow}

 \makeindex
 
\begin{document}
 \title{Periods of double EPW-sextics}
 \author{Kieran G. O'Grady\\\\
\lq\lq Sapienza\rq\rq Universit\`a di Roma}
\date{July 20 2014}
\thanks{Supported by PRIN 2010}
  \maketitle
 \section{Introduction}\label{sec:prologo}
 \setcounter{equation}{0}
 Let $V$ be a $6$-dimensional complex vector space.   Let $\lagr\subset\Gr(10,\bigwedge^3 V)$ be the symplectic Grassmannian parametrizing subspaces which are lagrangian for the symplectic form given by wedge-product.
  Given $A\in\lagr$ we let
   \begin{equation*}
Y_A:=\{[v]\in\PP(V) \mid A\cap (v\wedge\bigwedge^2 V)\not=\{0\}\}.
\end{equation*}
Then $Y_A$ is a degeneracy locus and hence it is naturally a subscheme of $\PP(V)$. 
For certain pathological choices of $A$ we have $Y_A=\PP(V)$: barring those cases    $Y_A$   is a sextic hypersurface named {\it EPW-sextic}. An EPW-sextic comes equipped with a double cover~\cite{ogdoppio} 
\begin{equation}\label{xaya}
f_A\colon X_A\to Y_A.
\end{equation}
 $X_A$ is what we call a {\it double EPW-sextic}. There is an open dense subset $\lagr^0\subset\lagr$ 
 parametrizing smooth double EPW-sextics - these $4$-folds 
are hyperk\"ahler (HK) deformations of the Hilbert square of a $K3$ (i.e.~the blow-up of the diagonal in the symmetric product of a $K3$ surface), see~\cite{ogduca}. By varying $A\in\lagr^0$ one gets a locally versal family of HK varieties - one of the five known such families in dimensions greater than $2$, see~\cite{beaudon,debvoi,iliran1,iliran2,llsv} for the construction of the other families. 
The complement  of $\lagr^0$ in $\lagr$ is the union of two prime divisors, $\Sigma$ and $\Delta$; the former consists of those $A$ containing a non-zero decomposable tri-vector,  the latter is defined in~\Ref{subsec}{divdel}. If $A$ is generic in $\Sigma$ then $X_A$ is singular along a $K3$ surface, see~\Ref{crl}{runrabbit}, if $A$ is generic in 
$\Delta$ then $X_A$ is singular at a single point whose tangent cone is isomorphic to the contraction of the zero-section of the cotangent sheaf of $\PP^2$, see Prop.~3.10 of~\cite{ogdoppio}.  
 By associating to $A\in\lagr^0$ the   Hodge structure on $H^2(X_A)$ one gets    
a  regular map of quasi-projective varieties~\cite{grhusa} 
\begin{equation}\label{perzero}
\cP^0\colon\lagr^0 \lra \DD_{\Lambda}
\end{equation}
where $\DD_{\Lambda}$ is a quasi-projective period domain, the quotient of a bounded symmetric domain of Type IV by the action of an arithmetic group, see~\Ref{subsec}{radici}.  Let $\DD_{\Lambda}^{BB}$ be the Baily-Borel compactification of $\DD_{\Lambda}$ and
\begin{equation}\label{perraz}
\cP\colon \lagr\dra \DD_{\Lambda}^{BB}
\end{equation}
the rational  map defined by~\eqref{perzero}.
The  map $\cP$ descends to the GIT-quotient of $\lagr$ for the natural action of $\PGL(V)$. More precisely: the action of $\PGL(V)$  on $\lagr$  is uniquely linearized and hence there is an unambiguous GIT quotient 
\begin{equation}
\gM:=\lagr//\PGL(V).
\end{equation}
Let  $\lagr^{\rm st},\lagr^{\rm ss}\subset\lagr$ be the loci of (GIT) stable and semistable points of $\lagr$.   By Cor.~2.5.1 of~\cite{ogmoduli} the open $\PGL(V)$-invariant subset $\lagr^0$ is contained in  $\lagr^{\rm st}$: we let 
\begin{equation}
\gM^0:=\lagr^0//\PGL(V).
\end{equation}
Then $\cP$ descends to   a rational map 
\begin{equation}\label{razper}
\gp\colon\gM\dra\DD_{\Lambda}^{BB}
\end{equation}
which is regular on $\gM^0$. By  Verbitsky's global Torelli Theorem and Markman's monodromy results the restriction of  $\gp$ to $\gM^0$ is injective, see Theorem~1.3 and Lemma~9.2 of~\cite{markman}. Since domain and codomain of the period map have the same dimension it follows that $\gp$ is a birational map.
In the present paper we will  be mainly concerned with the following problem: what is the indeterminacy locus of  $\gp$?  In order to state our main results we will go through a few more definitions.   Given  $A\in\lagr$  we let
\begin{equation}\label{eccoteta}
\Theta_A:=\{W\in\Gr(3, V)\mid \bigwedge^3 W\subset A\}.
\end{equation}
Thus $A\in\Sigma$ if and only if $\Theta_A\not=\es$. Suppose that 
  $W\in\Theta_A$: there is  a natural determinantal subscheme $C_{W,A}\subset\PP(W)$, see Subsect.~3.2 of~\cite{ogmoduli}, with the property that 
\begin{equation}
\supp C_{W,A}=\{[v]\in\PP(W) \mid \dim( A\cap(v\wedge\bigwedge^2 V))\ge 2\}.
\end{equation}
 $C_{W,A}$ is either a sextic curve or (in pathological cases) $\PP(W)$.   Let 
\begin{equation}\label{persestiche}
|\cO_{\PP(W)}(6)|\dashrightarrow \DD^{BB}_{\Phi} 
\end{equation}
be the  compactified period map where $\DD^{BB}_{\Phi} $ is the Baily-Borel compactification of the period space for  
$K3$ surfaces of degree $2$, see~\cite{shah}.
\begin{dfn}\label{dfn:igotico}
\begin{enumerate}
\item[(1)]
Let $\gM^{ADE}\subset \gM$    be the subset of points represented by $A\in\lagr^{ss}$ for which the following hold:
\begin{enumerate}
\item[(1a)]
The  orbit $\PGL(V)A$ is closed in $\lagr^{ss}$.
\item[(1b)]
For all $W\in\Theta_A$  we have  that $C_{W,A}$ is a sextic curve 
 with simple singularities. 
\end{enumerate}
\item[(2)]
Let $\gI\subset \gM$    be the subset of points represented by $A\in\lagr^{ss}$ for which the following hold:
\begin{enumerate}
\item[(2a)]
The  orbit $\PGL(V)A$ is closed in $\lagr^{ss}$.
\item[(2b)]
 There exists $W\in\Theta_A$  such that $C_{W,A}$ is either $\PP(W)$ or a sextic curve 
in the indeterminacy locus of~\eqref{persestiche}. 
\end{enumerate}
\end{enumerate}
\end{dfn}
Then $\gM^{ADE}$,  $\gI$ are respectively  open and closed subsets of $\gM$, and since every point of $\gM$ is represented by a single closed  $\PGL(V)$-orbit $\gI$ is in the complement of   $\gM^{ADE}$.  As is well-known the family of double EPW-sextics is analogous to the family of varieties of lines on a smooth cubic $4$-fold, and the period map for double EPW-sextics is analogous to the period map for cubic $4$-folds:  the subset $\gM^{ADE}$ is the analogue in our context of the moduli space of cubic $4$-folds with simple singularities, see~\cite{laza,looij}. 
 Below is the main result of the present paper.
\begin{thm}\label{thm:teorprinc}
The period map $\gp$ is regular away from $\gI$. Let $x\in(\gM\setminus\gI)$: then $\gp(x)\in\DD_{\Lambda}$ if and only if $x\in\gM^{ADE}$.
\end{thm}
\Ref{thm}{teorprinc} is the analogue of the result that the period map for cubic $4$-folds extends regularly to  the moduli space of cubic $4$-folds with simple singularities, and it maps it into the interior of the relevant Baily-Borel compactification, see~\cite{laza,looij}. 
The above result is a first step towards an understanding of the rational map $\gp\colon\gM\dra\DD_{\Lambda}^{BB}$. 
Such an understanding will eventually  include  a characterization of the image  of $(\lagr\setminus\Sigma)$. (Notice that if $A\in(\lagr\setminus\Sigma)$ then $A$ is a stable point (Cor.~2.5.1 of~\cite{ogmoduli}) and hence $[A]\in\gM^{ADE}$ because $\Theta_A$ is empty.) 
In the present paper we  will give a preliminary result in that direction. In order to state our result we will to introduce more notation. Let
\begin{equation}
\gN:=\Sigma//\PGL(V).
\end{equation}
(The generic point of $\Sigma$ is $\PGL(V)$-stable by Cor.~2.5.1 of~\cite{ogmoduli}, hence $\gN$ is a prime divisor of $\gM$.)   
 In~\Ref{subsec}{radneg} we will prove that   the set of Hodge structures in $\DD_{\Lambda}$  which have a $(1,1)$-root of negative square is the union of
four prime divisors named ${\mathbb S}_2^{\star}$, ${\mathbb S}'_2$, ${\mathbb S}''_2$, ${\mathbb S}_4$. 
\begin{thm}\label{thm:imagoper}
The restriction of the period map $\gp$  to $(\gM\setminus\gN)$ is an open embedding 
\begin{equation}\label{looicomp}
(\gM\setminus\gN)\hra (\DD_{\Lambda}
\setminus({\mathbb S}_2^{\star}\cup{\mathbb S}'_2\cup{\mathbb S}''_2\cup{\mathbb S}_4)).
\end{equation}
\end{thm}
Let us briefly summarize the main intermediate results of the paper and the proofs of~\Ref{thm}{teorprinc} and~\Ref{thm}{imagoper}.  In~\Ref{subsec}{radneg}  we will define the Noether-Lefschetz prime divisor   ${\mathbb S}_2^{\star}$ of $\DD_{\Lambda}$ and its closure  $\ov{\mathbb S}_2^{\star}$ in $\DD_{\Lambda}^{BB}$;  later we will prove that  the closure of $\cP(\Sigma)$ is equal to $\ov{\mathbb S}_2^{\star}$.  In~\Ref{subsec}{radneg}   we will show that the normalization of  $\ov{\mathbb S}_2^{\star}$ is equal to the Baily-Borel compactification $\DD^{BB}_{\Gamma}$ of the quotient of a bounded symmetric domain of Type IV modulo an arithmetic group 
and we will define    a natural finite map  $\rho\colon \DD^{BB}_{\Gamma}\to \DD^{BB}_{\Phi}$ where $\DD^{BB}_{\Phi}$ is as in~\eqref{persestiche}  - the map $\rho$ will play a key 
r\^ole in the proof of~\Ref{thm}{teorprinc}. In~\Ref{sec}{primest} we will prove that  $\cP$ is regular away from a certain closed subset  of  $\Sigma$ which has codimension $4$ in $\lagr$. The idea of the proof is the following. Suppose that $X_A$ is smooth and ${\bf L}\subset\PP(V)$ is a $3$-dimensional linear subspace such that  $f_A^{-1}(Y_A\cap{\bf L})$ is smooth: by Lefschetz' Hyperplane Theorem the periods of $X_A$ inject into the periods of $f_A^{-1}(Y_A\cap{\bf L})$.  This together with  Griffiths' Removable Singularity Theorem  gives that the period map extends regularly over the subset of $\lagr$ parametrizing those $A$ for which $f_A^{-1}(Y_A\cap{\bf L})$ has at most rational double points for generic ${\bf L}\subset\PP(V)$ as above.   We will prove that the latter condition holds away from the union of the   subsets of $\Sigma$  denoted    $\Sigma[2]$ and  $\Sigma_{\infty}$, see~\Ref{prp}{platone}. One gets the stated result because the codimensions in $\lagr$ of $\Sigma[2]$ and $\Sigma_{\infty}$ are  $4$ and $7$ respectively, see~\eqref{romalazio} and~\eqref{siginf}. 
 Let $\lagrhat\subset\lagr\times\DD_{\Lambda}^{BB}$ be the closure of the graph of $\cP$. Since  $\lagr$ is smooth the projection $p\colon\lagrhat \to \lagr$ is identified with the blow-up of the indeterminacy locus of $\cP$ and hence the exceptional set of $p$ is the support of the exceptional Cartier divisor of the blow-up.
 Let $\wh{\Sigma}\subset \lagrhat$ be the strict transform of $\Sigma$. The  results of~\Ref{sec}{primest}  described above  give that if $A$ is in the indeterminacy locus of $\cP$ (and hence $A\in\Sigma$) then 
\begin{equation}\label{almenodue}
\dim (p^{-1}(A)\cap\wh{\Sigma})\ge 2.
\end{equation}
 \Ref{sec}{perdisig} starts with an analysis of $X_A$ for generic $A\in\Sigma$: we will prove that it is singular along a $K3$ surface $S_A$ which is a double cover of $\PP(W)$ where $W$ is the unique element of $\Theta_A$ (unique because $A$ is generic in $\Sigma$) and that the blow-up of $X_A$ with center $S_A$ - call it $\wt{X}_A$ -  is a smooth HK variety deformation equivalent to  smooth double EPW-sextics, see~\Ref{crl}{runrabbit} and~\Ref{crl}{conichesuw}. It follows that $\cP(A)$ is identified with the weight-$2$ Hodge structure of  $\wt{X}_A$. Let $\zeta_A$ be the Poincar\'e dual of the exceptional divisor of the blow-up $\wt{X}_A\to X_A$. Then $\zeta_A$ is a $(-2)$-root of divisibility $1$ perpendicular to  the pull-back of $c_1(\cO_{Y_A}(1))$: it follows that  $\cP(A)\in\ov{\mathbb S}_2^{\star}$ and that $\cP(\Sigma)$ is dense in $\ov{\mathbb S}_2^{\star}$, see~\Ref{prp}{nerola}. 
  We will also define an  index-$2$ inclusion of integral Hodge structures $\zeta_A^{\bot}\hra H^2(S_A;\ZZ)$, see~\eqref{leuca}, and we will show that the inclusion may be identified with the value at $\cP(A)$ of the finite map $\rho\colon \DD^{BB}_{\Gamma}\to \DD^{BB}_{\Phi}$ mentioned above (this makes sense because $\rho$ is the map associated to an extension of lattices), see~\eqref{pescia}. Let 
\begin{equation}\label{sigtil}
\wt{\Sigma}:=  \{(W,A)\in\Gr(3,V)\times\lagr \mid \bigwedge^3 W\subset A\}.
\end{equation}
The natural forgetful map $\wt{\Sigma}\to\Sigma$ is birational (for general $A\in\Sigma$ there is a unique $W\in\Gr(3,V)$ such that $\bigwedge^3 W\subset A$); since the period map is regular at the generic point of $\Sigma$ it induces a rational map $\wt{\Sigma}\dra\ov{\mathbb S}_2^{\star}$ and hence a rational map to its normalization
\begin{equation}\label{checalore}
\wt{\Sigma}\dra\DD^{BB}_{\Gamma}.
\end{equation}
 Let $(W,A)\in\wt{\Sigma}$ and suppose that  $C_{W,A}$ is a sextic (i.e.~$C_{W,A}\not=\PP(W)$) and  the period map~\eqref{persestiche} is regular at $C_{W,A}$: the relation described above  between the Hodge structures  on $\zeta_A^{\bot}$ and  $H^2(S_A)$ gives that Map~\eqref{checalore} is regular at $(W,A)$. 
Now let $x\in(\gM\setminus\gI)$ and suppose that $x$ is in the indeterminacy locus of the rational period map $\gp$. One reaches a contradiction arguing as follows. Let $A\in\lagr$ be semistable with orbit closed in $\lagr^{ss}$ and representing $x$. Results of~\cite{ogmoduli} and~\cite{ogtasso} give that $\dim\Theta_{A}\le 1$; this result combined with  the  regularity of~\eqref{checalore}   at all $(W,A)$ with $W\in\Theta_{A}$ gives that $\dim (p^{-1}(A)\cap\wh{\Sigma})\le 1$: that contradicts~\eqref{almenodue} and hence proves that $\gp$ is regular at $x$ (it proves also the last clause in the statement of~\Ref{thm}{teorprinc}). 
In~\Ref{sec}{imagoper} we will prove~\Ref{thm}{imagoper}. The main ingredients of the proof  that $\gp$ maps $(\gM\setminus\gN)$ into the right-hand side of~\eqref{looicomp}  are Verbitsky's Global Torelli Theorem and our knowledge of degenerate EPW-sextics whose periods fill out open dense substes of the divisors ${\mathbb S}_2^{\star}$, ${\mathbb S}'_2$, 
${\mathbb S}''_2$ and ${\mathbb S}_4$. Here \lq\lq degenerate\rq\rq means that we have a hyperk\"ahler deformation of the Hilbert square of a $K3$ and   a map $f\colon X\to\PP^5$: while $X$ is \emph{not} degenerate, the map is degenerate in the sense that it is not a double cover of its image, either it has  (some) positive dimensional fibers (as in the case of  ${\mathbb S}_2^{\star}$ that we discussed  above) or  it has degree (onto its image) strictly higher than $2$. Lastly we will prove that the restriction of $\gp$ to  $(\gM\setminus\gN)$ is open. Since $\DD_{\Lambda}$ is $\QQ$-factorial it will suffice to prove that the exceptional set of $\gp|_{(\gM\setminus\gN)}$ has codimension at least $2$. If $A\in\lagr^0$ then $\gp$ is open  at $[A]$ because $X_A$ is smooth (injectivity and surjectivity of the local period map). On the other hand there is a regular involution of $(\gM\setminus\gN)$ which lifts a regular involution of $\DD_{\Lambda}$, and  $(\Delta\setminus\Sigma)//\PGL(V)$ is not sent to itself by the involution: this will allow us to conclude that the exceptional set of $\gp|_{(\gM\setminus\gN)}$ has codimension at least $2$.

\medskip
\noindent
{\bf Acknowledgments.} Thanks go to Olivier Debarre for many suggestions  on how to improve the exposition. 
 \section{Preliminaries}\label{sec:prelim}
 \setcounter{equation}{0}
 \subsection{Local equation of EPW-sextics}\label{sec:vidaloca}
 \setcounter{equation}{0}
We will recall notation and results from~\cite{ogtasso}. Let $A\in\lagr$ and $[v_0]\in \PP(V)$. 
Choose a direct-sum decomposition
\begin{equation}\label{trasuno}
 V=[v_0]\oplus V_0. 
\end{equation}
We identify $V_0$ with the open affine  $(\PP(V)\setminus\PP(V_0))$ via the isomorphism
\begin{equation}\label{apertoaffine}
\begin{matrix}
 V_0 & \overset{\sim}{\lra} & \PP(V)\setminus\PP(V_0) \\
v & \mapsto &  [v_0+v].
\end{matrix}
\end{equation}
Thus $0\in V_0$ corresponds to $[v_0]$. Then
\begin{equation}\label{taylor}
Y_A\cap V_0=V(f_0+f_1+\cdots+f_6),\qquad f_i\in \Sym^i V_0^{\vee}.
\end{equation}
The following result collects together statements contained in Corollary 2.5 and Proposition 2.9 of~\cite{ogtasso}. 
\begin{prp}\label{prp:ipsing}
Keep assumptions and hypotheses as above.
Let    $k:=\dim (A\cap (v_0\wedge\bigwedge^2 V))$. 
\begin{itemize}
\item[(1)]
Suppose that there is no $W\in\Theta_A$ containing $v_0$. Then the following hold:
\begin{enumerate}
\item[(1a)]
$0=f_0=\ldots= f_{k-1}$ and $f_k\not=0$.
\item[(1b)]
Suppose that $k=2$ and hence $[v_0]\in Y_A(2)$. Then $Y_A(2)$ is  smooth two-dimensional at $[v_0]$.
\end{enumerate}
\item[(2)]
Suppose that there exists $W\in\Theta_A$ containing $v_0$. Then $0=f_0=f_1$.
\end{itemize}
\end{prp} 
Next we recall how one describes $Y_A\cap V_0$ under the following assumption:
\begin{equation}\label{trasdue}
\bigwedge^3 V_0 \cap A=\{0\}. 
\end{equation}
Decomposition~\eqref{trasuno} determines a direct-sum decomposition $\bigwedge^3 V=[v_0]\wedge\bigwedge^2 V_0\oplus \bigwedge^3 V_0$. 
We will identify  $\bigwedge^2 V_0$  with $v_0\wedge\bigwedge^2 V_0$ via 
\begin{equation}\label{iaquinta}
\begin{matrix}
\bigwedge^2 V_0 & \overset{\sim}{\lra} & v_0\wedge\bigwedge^2 V_0\\
\beta & \mapsto & v_0\wedge\beta
\end{matrix}
\end{equation}
By~\eqref{trasdue} the subspace $A$ is the graph of a linear map 
\begin{equation}\label{aprile}
\wt{q}_A\colon \bigwedge^2 V_0\to \bigwedge^3 V_0.
\end{equation}
Choose a volume-form  
\begin{equation}\label{volzero}
\vol_0\colon \bigwedge^5 V_0\overset{\sim}{\lra} \CC   
\end{equation}
 in order to identify $\bigwedge^3 V_0$ with $\bigwedge^2 V_0^{\vee}$. Then $\wt{q}_A$  is symmetric because $A\in\lagr$. Explicitly
\begin{equation}\label{giannibrera}
\wt{q}_A(\alpha)=\gamma\iff
(v_0\wedge\alpha+\gamma)\in A. 
\end{equation}
We let $q_A$ be the associated  quadratic form on $\bigwedge^2 V_0$. Notice that
\begin{equation}\label{nucqua}
\ker q_A=\{\alpha\in\bigwedge^2 V_0 \mid v_0\wedge\alpha \in A\cap (v_0\wedge\bigwedge^2 V)  \} 
\end{equation}
 is identified with $A\cap (v_0\wedge\bigwedge^2 V)$ via~\eqref{iaquinta}. 
Let  $v\in V_0$  and  $q_v$ be the Pl\"ucker quadratic form  defined by
\begin{equation}\label{quadpluck}
\begin{matrix}
\bigwedge^2 V_0 & \overset{q_v}{\lra} & \CC \\
\alpha & \mapsto & \vol_0( v\wedge \alpha\wedge\alpha)
\end{matrix}
\end{equation}
\begin{prp}[Proposition~2.18 of~\cite{ogtasso}]\label{prp:ipsdisc}
Keep notation and hypotheses as above, in particular~\eqref{trasdue} holds.  Then
\begin{equation}
Y_A\cap V_0=V(\det(q_A+q_v))\,.
\end{equation}
\end{prp}
Next we will state a hypothesis which ensures  existence of a decomposition~\eqref{trasuno}  such 
that~\eqref{trasdue} holds. First recall~\cite{ogbobomolov}
that  we have an isomorphism
\begin{equation}
\begin{matrix}
 \lagr & \overset{\delta}{\overset{\sim}{\lra}} & \lagrdual \\
 A & \mapsto & \Ann   A.
\end{matrix}
\end{equation}
Let $E\in\Gr(5,V)$;   then
 \begin{equation}\label{uaidelta}
\text{$E\in Y_{\delta(A)}$ if and only if 
$(\bigwedge^3 E)\cap A\not=\{0\}$.}
\end{equation}
(The EPW-sextic $Y_{\delta(A)}$ is the  dual of $Y_A$.) Thus there exists  a decomposition~\eqref{trasuno}   such that~\eqref{trasdue} holds if and only if $Y_{\delta(A)}\not=\PP(V^{\vee})$. The  proposition below follows at once from Claim~2.11 and Equation~(2.82) of~\cite{ogtasso}.
\begin{prp}\label{prp:alpiudue}
Let $A\in\lagr$ and suppose that $\dim\Theta_A\le 2$. Then 
\begin{equation*}
Y_A\not=\PP(V),\qquad Y_{\delta(A)}\not=\PP(V^{\vee}).
\end{equation*}
In particular there exists  a decomposition~\eqref{trasuno}   such that~\eqref{trasdue} holds. 
\end{prp}
Let $A\in\lagr$.  We will need to consider higher degeneracy loci attached to $A$. Let
\begin{equation}
Y_A[k]=\{[v]\in\PP(V)\mid \dim(A\cap (v\wedge\bigwedge^2 V))\ge k\}.
\end{equation}
Notice that $Y_A[0]=\PP(V)$ and $Y_A[1]=Y_A$. Moreover $A\in\Delta$ if and only if $Y_A[3]$ is not empty. 
We set
 \begin{equation}\label{yerre}
Y_A(k):=Y_A[k]\setminus Y_A[k+1].
\end{equation}
 \subsection{Explicit description  of double EPW-sextics}\label{sec:eqdoppio}
 \setcounter{equation}{0}
Throughout the  present subsection we will  assume that
 $A\in\lagr$ and  $Y_A\not=\PP(V)$. Let $f_A\colon X_A\to Y_A$ be the double cover of~\eqref{xaya}.  The following is an immediate consequence of the definition of $f_A$, see~\cite{ogdoppio}:
\begin{equation}\label{nonram}
\text{$f_A$ is a topological covering of degree $2$ over $Y_A(1)$.}
\end{equation}
Let $[v_0]\in Y_A$;  we will give    explicit equations for  a neighborhood of $f_A^{-1}([v_0])$ in $X_A$. We will assume throughout the   subsection that we are given a direct-sum decomposition~\eqref{trasuno} such that~\eqref{trasdue} holds.  We start  by introducing some notation. 
 Let $K:=\ker q_A$ and  
let  $J\subset\bigwedge^2 V_0$ be a maximal subspace over which $q_A$ is non-degenerate;  
 we have a direct-sum decomposition
\begin{equation}\label{donmilani}
\bigwedge^2 V_0=J\oplus K.
\end{equation}
Choose a basis of $\bigwedge^2 V_0$ adapted to Decomposition~\eqref{donmilani}. Let $k:=\dim K$. The Gram matrices of 
 $q_A$ and $q_v$ (for $v\in V_0$) relative to the chosen basis are given by
\begin{equation}\label{patrici}
M(q_A)=
\begin{pmatrix}
N_{J} & 0 \\
0 & 0_k
\end{pmatrix},
\quad
M(q_v)=
\begin{pmatrix}
Q_{J}(v) & R_{J}(v)^t \\
R_{J}(v) & P_{J}(v)
\end{pmatrix}.
\end{equation}
(We let $0_k$ be the $k\times k$ zero matrix.) 
Notice that $N_{J}$ is invertible and $q_0=0$;
thus there exist arbitrarily small  open (in the classical topology) neighborhoods $V_0'$ of $0$ in $V_0$  such that
$(N_{J}+Q_{J}(v))^{-1}$ exists for $v\in V'_0$. We let
\begin{equation}\label{amatriciana}
M_{J}(v):=P_{J}(v)-R_{J}(v)\cdot
(N_{J}+Q_{J}(v))^{-1}\cdot R_{J}(v)^t,
\qquad v\in V_0'\,.
\end{equation}
If $V_0'$ is sufficiently small we may write $(N_{J}+Q_{J}(v))=S(v)\cdot S(v)^t$ for all $v\in V_0'$ where $S(v)$  is an analytic  function of $v$ (for this we need $V_0'$ to be open in the classical topology) and  $S(v)$ is invertible  for all $v\in V_0'$.
Let $j:=\dim J$. For later use we record the following equality
\begin{equation}\label{congruenza}
\begin{pmatrix}
\scriptstyle 1_j & \scriptstyle 0 \\
\scriptstyle -R_J(v) S^{-1}(v)^t & \scriptstyle 1_k
\end{pmatrix}
\cdot
\begin{pmatrix}
\scriptstyle S(v)^{-1} & \scriptstyle 0 \\
\scriptstyle 0 & \scriptstyle 1_k
\end{pmatrix}
\cdot
\begin{pmatrix}
\scriptstyle N_J+Q_{J}(v) & \scriptstyle R_{J}(v)^t \\
\scriptstyle R_{J}(v) & \scriptstyle P_{J}(v)
\end{pmatrix}
\cdot
\begin{pmatrix}
\scriptstyle S^{-1}(v)^t & \scriptstyle 0 \\
\scriptstyle 0 & \scriptstyle 1_k
\end{pmatrix}
\cdot 
\begin{pmatrix}
\scriptstyle 1_j & \scriptstyle -S^{-1}(v)R_J(v)^t \\
\scriptstyle 0 & \scriptstyle 1_k
\end{pmatrix}
\scriptstyle =
\begin{pmatrix}
\scriptstyle 1_j & \scriptstyle 0 \\
\scriptstyle 0 & \scriptstyle M_J(v)
\end{pmatrix}
\end{equation}
Let  ${\bf X}_{J}\subset V_0'\times\CC^{k}$ be the closed subscheme whose ideal is generated by the entries of the matrices
\begin{equation}\label{eqespl}
M_{J}(v)\cdot\xi,\qquad \xi\cdot\xi^t-M_{J}(v)^c,
\end{equation}
where $\xi\in\CC^k$  is a column vector and $M_{J}(v)^c$ is the matrix of cofactors of $M_{J}(v)$. 
We identify $V'_0$ with an open neighborhood of $[v_0]\in \PP(V)$ via~\eqref{apertoaffine}.
Projection  defines a map $\phi\colon {\bf X}_{J}\to V(\det M_{J})$. By~\eqref{congruenza} we have $V(\det M_{J})=V'_0\cap Y_A$. 
\begin{prp}\label{prp:locdegsym}
Keep notation and assumptions as above. 
There exists  a commutative diagram 
\begin{equation*}
\xymatrix{  \\  (X_A,f_A^{-1}([v_0]))\ar^{f_A}[dr]  \ar^{\zeta}[rr] & & 
({\bf X}_{J},\,\phi^{-1}([v_0]))\ar_{\phi}[dl] \\
& (Y_A,[v_0]) &}
\end{equation*}
where the germs are in the analytic topology. Furthermore $\zeta$ is an isomorphism.
\end{prp} 
\begin{proof}
Let $[v]\in\PP(V)$: there is a canonical identification between $v\wedge \bigwedge^2 V$ and the fiber at $[v]$ of  $\Omega^3_{\PP(V)}(3)$, see for example Proposition~5.11 of~\cite{ogduca}. Thus we have an injection $\Omega^3_{\PP(V)}(3)\hra\bigwedge^3 V\otimes\cO_{\PP(V)}$.
Choose $B\in\lagr$ transversal to $A$. The direct-sum decomposition
$\bigwedge^3 V=A\oplus B$ gives rise to a  commutative diagram with exact rows
\begin{equation}\label{spqr}
\begin{array}{ccccccccc}
0 & \to & \Omega^3_{\PP(V)}(3) &\mapor{\lambda_A}& A^{\vee}\otimes\cO_{\PP(V)} & \lra & i_{*}\zeta_A
&
\to & 0\\
 & & \mapver{\mu_{A,B}}& &\mapver{\mu^{t}_{A,B}} &
&
\mapver{\beta_{A}}& & \\
0 & \to & A\otimes\cO_{\PP(V)}& \mapor{\lambda_A^{t}}& \bigwedge^3 T_{\PP(V)}(-3) & \lra &
Ext^1(i_{*}\zeta_A,\cO_{\PP(V)}) & \to & 0
\end{array}
\end{equation}
(As  suggested by our notation the map $\beta_A$ does not depend on the choice of $B$.) Choosing $B$ transverse to $v_0\wedge\bigwedge^2 V$ we can assume that $\mu_{A,B}(0)$ (recall that $(\PP(V)\setminus\PP(V_0))$ is identified with $V_0$ via~\eqref{apertoaffine} and that $[v_0]$ corresponds to $0$) is an isomorphism. Then there exist   arbitrarily small open (classical topology) neighborhoods $\cU$ of $0$ such that $\mu_{A,B}(v)$ is an isomorphism for all $v\in\cU$. The map $\lambda_A\circ \mu_{A,B}^{-1}$ is symmetric because $A$ is lagrangian. Choose a basis of $A$ and let $L(v)$ be the Gram matrix of  $\lambda_A\circ \mu_{A,B}^{-1}(v)$   with respect to the chosen basis. Let $L(v)^c$ be the matrix of cofactors of $L(v)$. Claim~1.3  of~\cite{ogdoppio} gives an embedding
\begin{equation}\label{locdoppio}
f_A^{-1}(\cU\cap Y_A)\hra\cU\times\aff^{10} 
\end{equation}
with image  the closed subscheme whose ideal is generated   by the entries of the matrices
\begin{equation}\label{equazelle}
L(v)\cdot\xi\,\qquad \xi\cdot\xi^t-L(v)^c.
\end{equation}
(Here  $\xi$ is a $10\times 1$-matrix whose entries are coordinates on $\aff^{10}$.) We will denote the above subscheme by $V(L(v)\cdot\xi,\  \xi\cdot\xi^t-L(v)^c)$.
Under this embedding the restriction of $f_A$ to 
$f_A^{-1}(\cU\cap Y_A)$ gets identified with the restriction of the projection $\cU\times\aff^{10}\to\cU$. Let $G\colon \cU\to\GL_{10}(\CC)$ be an analytic map and for $v\in\cU$ let $H(v):=G^t(v)\cdot L(v)\cdot G(v)$.  The automorphism of $\cU\times\aff^{10}$ given by $(v,\xi)\mapsto (v,G(v)^{-1}\xi)$ restricts to an isomorphism
\begin{equation}
V(L(v)\cdot\xi,\  \xi\cdot\xi^t-L(v)^c) \overset{\sim}{\lra}  V(H(v)\cdot\xi,\  \xi\cdot\xi^t-H(v)^c).
\end{equation}
In other words we are free to replace $L$ by an arbitrary congruent matrix function. 
Let 
\begin{equation}\label{arbie}
\begin{matrix}
\bigwedge^2 V_0 & \overset{\phi_{v_0+v}}{\lra} &   (v_0+v)\wedge\bigwedge^2 V  \\
\alpha & \mapsto & (v_0+v)\wedge\alpha
\end{matrix}
\end{equation}
 A straightforward computation gives that 
 \begin{equation}
\phi_{v_0+v}^t\circ \mu_{A,B}^t(v)\circ(\lambda_A(v)\circ \mu_{A,B}^{-1}(v))\circ \mu_{A,B}(v)\circ\phi_{v_0+v}=\wt{q}_A+\wt{q}_{v},\qquad v\in \cU. 
\end{equation}
Thus the Gram matrix $M(q_A+q_v)$ is congruent to $L(v)$ and hence we have an embedding~\eqref{locdoppio} with image   $V(M(q_A+q_v)\cdot\xi,\  \xi\cdot\xi^t-M(q_A+q_v)^c)$. On the other hand~\eqref{congruenza} shows that $M(q_A+q_v)$ is congruent to the matrix
\begin{equation}
E(v):=
\begin{pmatrix}
\scriptstyle 1_j & \scriptstyle 0 \\
\scriptstyle 0 & \scriptstyle M_J(v)
\end{pmatrix}
\end{equation}
Thus  we have an embedding~\eqref{locdoppio} with image   $V(E(v)\cdot\xi,\  \xi\cdot\xi^t-E(v)^c)$. A straightforward computation shows that the latter subscheme is isomorphic to ${\bf X}_J\cap (\cU\times\CC^k)$.
\end{proof}
 \subsection{The subscheme $C_{W,A}$}\label{sec:ciwloc}
 \setcounter{equation}{0}
Let $(W,A)\in\wt{\Sigma}$. For the definition of the subscheme $C_{W,A}\subset\PP(W)$ we refer to Subs.~3.2 
of~\cite{ogmoduli}.  
\begin{dfn}\label{dfn:malvagio}
Let $\cB(W,A)\subset\PP(W)$ be the set of $[w]$ such that one of the following holds:
\begin{itemize}
\item[(1)]
There exists $W'\in(\Theta_A\setminus\{W\})$ containing $w$.
\item[(2)]
$\dim(A\cap (w\wedge\bigwedge^2 V) \cap (\bigwedge^2 W\wedge V))\ge 2$.
\end{itemize}
\end{dfn}
The following result is obtained by pasting together  Proposition 3.3.6 and Corollary 3.3.7 of~\cite{ogmoduli}.
\begin{prp}\label{prp:cnesinerre}
Let $(W,A)\in\wt{\Sigma}$. Then the following hold:
\begin{enumerate}
\item
  $C_{W,A}$ is a  smooth curve at $[v_0]$ if and only if $\dim (A\cap (v_0\wedge\bigwedge^2 V))=2$ and $[v_0]\notin \cB(W,A)$.
\item
$C_{W,A}=\PP(W)$ if and only if  $\cB(W,A)=\PP(W)$.  
\end{enumerate}
\end{prp}
 \subsection{The divisor $\Sigma$}\label{sec:ancorasig}
 \setcounter{equation}{0}
Given $d\ge 0$ we let $\wt{\Sigma}[d]\subset\wt{\Sigma}$  be 
\begin{equation}
\wt{\Sigma}[d]:= \{(W,A)\in\wt{\Sigma}\mid \dim(A\cap (\bigwedge^2 W\wedge V))\ge d+1\}.
\end{equation}
Notice that $\wt{\Sigma}:=\wt{\Sigma}[0]$. 
Let 
\begin{equation}
\Gr(3,V)\times\lagr\overset{\pi}{\lra} \lagr
\end{equation}
be projection and $\Sigma[d]:=\pi(\wt{\Sigma}[d])$.
Notice that $\Sigma:=\Sigma[0]$. 
Proposition 3.1 of~\cite{ogtasso} gives that
\begin{equation}\label{romalazio}
\cod(\Sigma[d],\lagr)=(d^2+d+2)/2.
\end{equation}
Let
\begin{equation}
\Sigma_{+}:=  \{A\in\Sigma\mid \Card(\Theta_A)>1\}.
\end{equation}
Proposition 3.1 of~\cite{ogtasso} gives that  $\Sigma_{+}$ is a constructible subset of $\lagr$ and
\begin{equation}\label{codsigpiu}
\cod(\Sigma_{+},\lagr)=2.
\end{equation}
We claim that
\begin{equation}\label{singsig}
\sing\Sigma=\Sigma_{+}\cup\Sigma[1].
\end{equation}
In fact  $(\ov{\Sigma}_{+}\setminus \Sigma_{+})\subset\Sigma[1]$ by Equation~(3.19)  of~\cite{ogtasso} and hence~\eqref{singsig} follows from Proposition~3.2 of~\cite{ogtasso}.
We let
\begin{equation}\label{siginf}
\Sigma_{\infty}:=  \{A\in\lagr\mid \dim\Theta_A>0\}.
\end{equation}
Theorem 3.37 and Table 3 of~\cite{ogtasso} give the following:
\begin{equation}\label{codinf}
\cod(\Sigma_{\infty},\lagr)=7.
\end{equation}
 \subsection{The divisor $\Delta$}\label{subsec:divdel}
 \setcounter{equation}{0}
Let 
\begin{equation}
\Delta:=    \{A\in\lagr\mid \text{$\exists [v]\in\PP(V)$ such that $\dim (A\cap (v\wedge\bigwedge^2 V))\ge 3$} \}.
\end{equation}
A dimension count gives that $\Delta$ is a prime divisor in $\lagr$, see~\cite{ogdoppio}.
Let
\begin{equation}
\wt{\Delta}(0):=\{([v],A)\in\PP(V)\times\lagr \mid \dim ( A\cap (v\wedge\bigwedge^2 V))=3\}.
\end{equation}
The following result will be handy.
\begin{prp}\label{prp:delinf}
Let $A\in\lagr$ and suppose that $\dim Y_A[3]>0$. Then
$A\in(\Sigma_{\infty}\cup\Sigma[2])$.
\end{prp}
\begin{proof}
By contradiction. Thus we assume that $\dim Y_A[3]>0$ and 
$A\notin(\Sigma_{\infty}\cup\Sigma[2])$. By hypothesis there exists an irreducible component $C$ of $Y_A[3]$ of strictly positive dimension. Let $[v]\in C$ be generic. We claim that one of the following holds:
\begin{itemize}
\item[(a)]
There exist  distinct $W_1([v]),W_2([v])\in\Theta_A$ containing $v$.
\item[(b)]
There exists $W([v])\in\Theta_A$ containing  $v$ and such that
\begin{equation}\label{caffe}
\dim  A\cap S_{W([v])}\cap (v\wedge\bigwedge^2 V) \ge 2.
\end{equation}
\end{itemize}
In fact assume first that $\dim(A\cap (v\wedge\bigwedge^2 V))=3$ for  $[v]$ in an open dense $C^0\subset C$. We may assume that $C^0$ is smooth; then we have an embedding $\iota\colon C^0\hra\wt{\Delta}(0)$ defined by mapping $[v]\in C^0$ to $([v],A)$. Let $[v]\in C^0$: the 
 differential of the projection $\wt{\Delta}(0)\to\lagr$ at $([v],A)$ is \emph{not} injective because it vanishes on $\im d\iota([v])$.
By Corollary~3.4 and Proposition~3.5 of~\cite{ogdoppio} we get that one of Items~(a), (b) above holds. 
Now assume that $\dim(A\cap (v\wedge\bigwedge^2 V))>3$ for generic $[v]\in C$ (and hence for all $[v]\in C$). Let notation be as in the proof of  Proposition 3.5 of~\cite{ogdoppio}; then ${\bf K}\cap \Gr(2,V_0)$ is a zero-dimensional (if it has strictly positive dimension then $\dim\Theta_A>0$ and hence $A\in\Sigma_{\infty}$ against our assumption)  scheme of length $5$.  It follows that either Item~(a) holds (if ${\bf K}\cap \Gr(2,V_0)$ is not a single point) or Item~(b) holds (if ${\bf K}\cap \Gr(2,V_0)$ is a single point ${\bf p}$ and hence the tangent space of ${\bf K}\cap \Gr(2,V_0)$ at ${\bf p}$ has dimension at least $1$). Now we are ready to reach a contradiction. First suppose that Item~(a)  holds. Since $\Theta_A$ is finite there exist distinct $W_1,W_2\in\Theta_A$ such that $C\subset(\PP(W_1)\cap \PP(W_2))$. Thus $\dim(W_1\cap W_2)=2$ and hence the line
\begin{equation}
\{W\in\Gr(3,V)\mid (W_1\cap W_2)\subset W\subset (W_1+W_2)\}
\end{equation}
is contained in $\Theta_A$, that is a contradiction. Now suppose that Item~(b) holds. Since $\Theta_A$ is finite there exists $W\in\Theta_A$ such that $C\subset \PP(W)$ and
\begin{equation}\label{cappuccino}
\dim  A\cap S_{W}\cap (v\wedge\bigwedge^2 V)\ge 2\qquad \forall [v]\in C.
\end{equation}
 Since  $A\notin\Sigma[2]$ we have $\dim(A\cap (\bigwedge^2 W\wedge V))=2$. Let $\{w_1,w_2,w_3\}$ be a basis of $W$; then
\begin{equation}
A\cap (\bigwedge^2 W\wedge V)=\la w_1\wedge w_2\wedge w_3,\, \beta \ra.
\end{equation}
Let $\ov{\beta}$ be the image  of $\beta$ under the quotient map $(\bigwedge^2 W\wedge V)\to (\bigwedge^2 W\wedge V)/\bigwedge^3 W$. Then 
\begin{equation}
\ov{\beta}\in\bigwedge^2 W\wedge(V/W)\cong \Hom(W,V/W).
\end{equation}
(We choose a volume form on $W$ in order to define the isomorphism above.)
By~\eqref{cappuccino}  the kernel of $\ov{\beta}$ (viewed as a map $W\to(V/W)$) contains all $v$ such that $[v]\in C$.
Thus $\ov{\beta}$  has rank $1$. It follows that $\beta$ is decomposable: $\beta\in\bigwedge^3 W'$ where $W'\in\Theta_A$ and $\dim W\cap W'=2$. Then $\Theta_A$ contains the line in $\Gr(3,V)$ joining $W$ and $W'$: that is a contradiction.
\end{proof}
 \subsection{Lattices and periods}\label{subsec:radici}
 \setcounter{equation}{0}
 Let $L$ be an even lattice: we will denote by $(,)$ the bilinear symmetric  form on $L$ and for $v\in L$ we let $v^2:=(v,v)$.  For a ring $R$  we let  $L_{R}:=L\otimes_{\ZZ} R$ and we let $(,)_{R}$ be the $R$-bilinear symmetric form on $L_R$ obtained from $(,)$ by extension of scalars. 
Let $L^{\vee}:=\Hom(L,\ZZ)$. The bilinear form defines an embedding $L\hra L^{\vee}$: the quotient $D(L):=L^{\vee}/L$ is the {\it discriminant group} of $L$.  Let $0\ne v\in L$ be primitive i.e.~$L/\la v\ra$ is torsion-free. The {\it divisibility} of $v$ is the positive generator of $(v,L)$ and is denoted  by $\divisore(v)$; we let $v^{*}:=v/\divisore(v)\in D(L)$.
The group $O(L)$ of isometries of $L$ acts naturally on $D(L)$.  The {\it stable} orthogonal group is equal to
\begin{equation}
\wt{O}(L):=\ker(O(L)\to D(L)).
\end{equation}
We let ${\bf q}_L\colon D(L)\to \QQ/2\ZZ$ and ${\bf b}_L\colon D(L)\times D(L)\to \QQ/\ZZ$ be the discriminant quadratic-form and symmetric bilinear form respectively, see~\cite{nikulin}.
 The following criterion of Eichler will be handy.
\begin{prp}[Eichler's Criterion, see Prop.~3.3 of~\cite{ghs-abtn}]\label{prp:criteich}
Let $L$ be an even lattice which contains $U^2$ (the direct sum of two hyperbolic planes). Let $v_1,v_2\in L$ be non-zero and primitive. There exists $g\in\wt{O}(L)$ such that $g v_1=v_2$ if and only if $v_1^2=v_2^2$ and $v_1^{*}=v_2^{*}$. 
\end{prp}
Now suppose  that $L$ is an even lattice of signature $(2,n)$. Let
\begin{equation}\label{domper}
\Omega_L:=  \{[\sigma]\in \PP(L_\CC)\mid (\sigma,\sigma)_\CC=0,\quad 
(\sigma,\ov{\sigma})_\CC>0\}.
\end{equation}
(Notice that the isomorphism class of $\Omega_L$ depends on $n$ only.) Then $\Omega_L$ is the union of two disjoint bounded symmetric domains of Type IV  on which 
$O(L)$ acts. By  Baily and Borel's fundamental results the quotient
\begin{equation}
\DD_L:=\wt{O}(L)\backslash\Omega_L.
\end{equation}
 is  quasi-projective.
\begin{rmk}\label{rmk:compdidi}
 Suppose that $v_0\in L$ has  square $2$. The reflection
\begin{equation}\label{rifdue}
\begin{matrix}
L & \overset{R_{v_0}}{\lra} & L \\
v & \mapsto & v-(v,v_0) v_0
\end{matrix}
\end{equation}
belongs to the stable orthogonal group. We claim that  $R_{v_0}$  exchanges  the two connected components of  $\Omega_L$. In fact let $M\subset L_{\RR}$ be a positive definite subspace of maximal dimension (i.e.~$2$) containing $v_0$.
If  $[\sigma]\in \Omega_L\cap(M_{\CC})$ then $R_{v_0}([\sigma])=[\ov{\sigma}]$: this proves our claim because conjugation interchanges the two connected components of $\Omega_L$.  
It follows that if  $L$ contains a vector of square $2$ then  $\DD_L$ is connected. 
\end{rmk}
  Let us examine  the lattices of interest to us. Let $J,M,N$ be three copies of the hyperbolic plane $U$, let $E_8(-1)$ be the unique unimodular negative definite even lattice of rank $8$ and $(-2)$ the rank-$1$ lattice with generator of square $(-2)$. Let 
\begin{equation}\label{lambdatilde}
 \wt{\Lambda}:=J\oplus M\oplus N \oplus E_8(-1)^2\oplus(-2) \cong U^3\oplus E_8(-1)^2\oplus(-2).
\end{equation}
 If $X$ is a HK manifold deformation equivalent to the Hilbert square of a $K3$ then $H^2(X;\ZZ)$ equipped with the Beauville-Bogomolov quadratic form is isometric to $\wt{\Lambda}$. A vector in $\wt{\Lambda}$ of square $2$ has divisibility $1$: it follows from~\Ref{prp}{criteich}  that any two vectors in $\wt{\Lambda}$ of square $2$ are $O(\wt{\Lambda})$-equivalent and hence the isomorphism class of $v^{\bot}$ for $v^2=2$ is independent of $v$. 
We choose  $v_1\in J$ of square $2$ and let $\Lambda:=v_1^{\bot}$. Then
\begin{equation}\label{isomteta}
\Lambda\cong U^2\oplus E_8(-1)^2\oplus(-2)^2.
\end{equation}
We get an inclusion $\wt{O}(\Lambda)< O(\wt{\Lambda})$ by associating to $g\in \wt{O}(\Lambda)$ the unique $\wt{g}\in  O(\wt{\Lambda})$ which is the identity on $\ZZ v_1$ and restricts to $g$ on $v_1^{\bot}$  (such a lift  exists  because $g\in\wt{O}(\Lambda)$). 
 Now suppose that $X$ is  a $HK$ manifold deformation equivalent to the Hilbert square of $K3$ and that 
 $h\in H^{1,1}_{\ZZ}(X)$ has square $2$. Since there is a single $O(\wt{\Lambda})$-orbit of square-$2$ vectors there exists an isometry 
 \begin{equation}
\psi\colon H^2(X;\ZZ)\overset{\sim}{\lra}\wt{\Lambda},\qquad \psi(h)=v_1.
\end{equation}
   Such an isometry is a {\it marking} of $(X,h)$. If $H$ is a divisor  on $X$ of square $2$ a marking of $(X,H)$ is a marking of $(X,c_1(\cO_X(H)))$. Let   $\psi_{\CC}\colon H^2(X;\CC)\to \wt{\Lambda}_{\CC}$ be the $\CC$-linear extension of $\psi$. Since $h$ is of type $(1,1)$ we have that $\psi_{\CC}(H^{2,0})\in v_1^{\bot}$. Well-known properties of the Beauville-Bogomolov quadratic form give that  $\psi_{\CC}(H^{2,0})\in\Omega_{\Lambda}$. 
Any two markings of $(X,h)$ differ by the action of an element of $\wt{O}(\Lambda)$. It follows that the equivalence class 
\begin{equation}\label{eccoperi}
\Pi(X,h):=[\psi_{\CC}H^{2,0}]\in\DD_{\Lambda}
\end{equation}
 is well-defined i.e.~independent of the marking: that is the \emph{period point} of $(X,h)$. Since the lattice $\Lambda$  contains vectors of square $2$ the quotient $\DD_{\Lambda}$ is irreducible by~\Ref{rmk}{compdidi}. 
The discriminant group and discriminant quadratic form of $\Lambda$ are described as follows. Let $e_1$ be a generator of
$v_1^{\bot}\cap J$ and let $e_2$ be a generator of the last summand of~\eqref{lambdatilde}:
\begin{equation}
\ZZ e_1=v_1^{\bot}\cap J,\qquad \ZZ e_2=(-2).
\end{equation}
Then $-2=e^2_1=e^2_2$, $(e_1,e_2)=0$ and $2=\divisore_{\Lambda}(e_1)=\divisore_{\Lambda}(e_2)$: here we denote by $\divisore_{\Lambda}(e_i)$ the divisibility of $e_i$ as element of $\Lambda$, one should notice that the divisibility of $e_1$ in $\wt{\Lambda}$ is $1$ (not $2$) while the divisibility of $e_2$ in $\wt{\Lambda}$ is $2$ (equal to the divisibility  of $e_2$ in $\Lambda$). In particular $e_1/2$ and $e_2/2$ are order-$2$ elements of $D(\Lambda)$. We have   the following:
\begin{equation}\label{tuttosudi}
\qquad\qquad
\begin{matrix}
\ZZ/(2)\oplus \ZZ/(2)  & \overset{\sim}{\lra} & D(\Lambda) \\
([x], [y]) & \mapsto & x(e_1/2)+y (e_2/2)
\end{matrix}
\qquad
q_{\Lambda}(x(e_1/2)+y (e_2/2))\equiv  -\frac{1}{2}x^2 -\frac{1}{2}y^2 \pmod{2\ZZ}
\end{equation}
In particular we get that 
\begin{equation}\label{indicegamma}
[O(\Lambda):\wt{O}(\Lambda)]=2.
\end{equation}
Let $\iota\in O(\Lambda)$ be the involution  characterized by 
\begin{equation}\label{vipera}
\iota(e_1)=e_2,\quad \iota(e_2)=e_1,\quad \iota|_{\{e_1,e_2\}^{\bot}}=\Id _{\{e_1,e_2\}^{\bot}}\,.
\end{equation}
Then $\iota\notin \wt{O}(\Lambda)$. Since $[O(\Lambda):\wt{O}(\Lambda)]=2$  we get that $\iota$
induces a non-trivial involution 
\begin{equation}\label{invoper}
\ov{\iota}\colon\DD_{\Lambda}^{BB}\to\DD_{\Lambda}^{BB}.
\end{equation}
The geometric counterpart of $\ov{\iota}$ is given by the involution $\delta\colon{\mathfrak M}\to{\mathfrak M}$ induced by the map
\begin{equation}\label{specchio}
\begin{matrix}
 \lagr & \overset{\delta_V}{\overset{\sim}{\lra}} & \lagrdual \\
 A & \mapsto & \delta_V(A)=\Ann   A.
\end{matrix}
\end{equation}
(The geometric meaning of $\delta_V(A)$: for generic $A\in\lagr$ the dual of $Y_A$ is equal  to $Y_{\delta_V(A)}$.)
 In~\cite{ogbobomolov} we proved that
 \begin{equation}\label{idiota}
\ov{\iota}\circ \gp=\gp\circ\delta. 
\end{equation}
 \subsection{Roots of $\Lambda$}\label{subsec:radneg}
 \setcounter{equation}{0}
Let  $v_0\in\Lambda$ be  \emph{primitive} and let $v_0^2=-2d\not=0$: then $v_0$ is a \emph{root}  if the reflection
\begin{equation}
\begin{matrix}
\Lambda_{\QQ} & \overset{R}{\lra} & \Lambda_{\QQ} \\
v & \mapsto & v+\frac{(v,v_0)v_0}{d}
\end{matrix}
\end{equation}
is integral, i.e.~$R(\Lambda)\subset\Lambda$. We record the square of $v_0$ by stating that $v_0$ is $(-2d)$-root.  Notice that if $v_0^2=\pm 2$ then $v_0$ is a root.
In particular $e_1$ and $e_2$ are $(-2)$-roots of $\Lambda$. Let 
\begin{equation}
e_3\in M,\qquad e_3^2=-2.
\end{equation}
 Notice that $e_3\in\Lambda$ and hence  it is a $(-2)$-root of $\Lambda$. Since $(e_1+e_2)^2=-4$ and $\divisore(e_1+e_2)=2$ we get that   $(e_1+e_2)$ is a  $(-4)$-root of $\Lambda$. 
\begin{prp}\label{prp:orbrad}
The set of negative roots of $\Lambda$  breaks up into $4$  orbits for the action of $\wt{O}(\Lambda)$, namely the  orbits of $e_1$, $e_2$,   $e_3$ and $(e_1+e_2)$. 
\end{prp}
\begin{proof}
First let us prove that the  orbits of $e_1$, $e_2$,   $e_3$ and $(e_1+e_2)$ are pairwise disjoint. Since $-2=e_1^2=e_2^2=e_3^2$ and $(e_1+e_2)^2=-4$ the orbits  of $e_1$, $e_2$ and   $e_3$  are disjoint from that of $(e_1+e_2)$.   
We have $\divisore_{\Lambda}(e_3)=1$ and hence $e_3^{*}=0$.  Since  $e_1^{*}$,   $e_2^{*}$  and $e_3^{*}$  are pairwise distinct elements of $D(\Lambda)$ it follows that the orbits of $e_1$, $e_2$,   $e_3$ are pairwise disjoint. Now let $v_0\in\Lambda$ be a negative root. Since  $D(\Lambda)$ is $2$-torsion $\divisore(v_0)\in\{1,2\}$: it follows that $v_0$ is either a $(-2)$-root or a $(-4)$-root, and in the latter case $\divisore(v_0)=2$. Suppose first that $v_0$ is a $(-2)$-root. If $\divisore_{\Lambda}(v_0)=1$ then $v_0^{*}=0$ and hence $v_0$ is in the orbit of $e_3$ by~\Ref{prp}{criteich}. If $\divisore_{\Lambda}(v_0)=2$ then $v^{*}\in\{e_1^{*},e_2^{*}\}$  because  $q_{\Lambda}(e_1^{*}+ e_2^{*})\equiv -1\not\equiv -1/2\pmod{2\ZZ}$: it follows from~\Ref{prp}{criteich} that $v_0$ belongs either to the $\wt{O}(\Lambda)$-orbit of $e_1$ or to that of $e_2$. Lastly suppose that $v_0$ is a $(-4)$-root. Since $\divisore(v_0)=2$ we have $q_{\Lambda}(v_0^{*})=-1$ and hence $v_0^{*}=e_1/2+e_2/2$: it follows from~\Ref{prp}{criteich} that $v_0$ belongs  to the $\wt{O}(\Lambda)$-orbit of $(e_1+e_2)$.
\end{proof}
Let  $\kappa\colon\Omega_{\Lambda}\to\DD_{\Lambda}$ be the quotient map.  Let
 \begin{equation}\label{essetiti}
{\mathbb S}_2':=\kappa  ( e_1^{\bot}\cap\Omega_{\Lambda}),\quad
{\mathbb S}_2'':=\kappa ( e_2^{\bot}\cap\Omega_{\Lambda}),\quad
{\mathbb S}_2^{\star}:=\kappa (e_3^{\bot}\cap\Omega_{\Lambda}),\quad
{\mathbb S}_4:=\kappa  ( (e_1+e_2)^{\bot}\cap\Omega_{\Lambda}).
\end{equation}
\begin{rmk}\label{rmk:essirr}
Let $i=1,2,3$: then $e_i^{\bot}\cap\Omega_{\Lambda}$ has two connected components - see~\Ref{rmk}{compdidi}. 
Let $v_0\in N$ (we refer to~\eqref{lambdatilde}) of square $2$. Then $(v_0,e_i)=0$ for 
$i=1,2,3$ and hence Reflection~\eqref{rifdue} exchanges the two connected components of  $e_i^{\bot}\cap\Omega_{\Lambda}$ for $i=1,2,3$ and also the two connected components of  $(e_1+e_2)^{\bot}\cap\Omega_{\Lambda}$ . 
It follows that   each of ${\mathbb S}_2'$, ${\mathbb S}_2''$, ${\mathbb S}_2^{\star}$ and
 ${\mathbb S}_4$ is a prime divisor in $\DD_{\Lambda}$. 
\end{rmk}  
Let $\ov{\iota}$ be the involution given by~\eqref{invoper}: then
\begin{equation}\label{azionesse}
\ov{\iota}({\mathbb S}_2^\star) =  {\mathbb S}_2^\star,\quad 
\ov{\iota}({\mathbb S}'_2) =  {\mathbb S}''_{2},\quad
\ov{\iota}({\mathbb S}''_2) =  {\mathbb S}'_{2},\quad
\ov{\iota}({\mathbb S}_4) =  {\mathbb S}_4.
\end{equation}
We will describe the normalization of ${\mathbb S}_2^{\star}$ and we will show that it is a finite cover of  the period space for $K3$ surfaces of degree $2$. Let $v_3$ be a generator of $e_3^{\bot}\cap M$. Let
\begin{equation}\label{tilgam}
\wt{\Gamma}:=e_3^{\bot}=J\oplus \ZZ v_3\oplus N\oplus E_8(-1)^2\oplus \ZZ e_2\cong
U\oplus (2)\oplus U\oplus E_8(-1)^2\oplus (-2)
\end{equation}
and
\begin{equation}
\Gamma:=e_3^{\bot}\cap\Lambda=\ZZ e_1\oplus \ZZ v_3\oplus N\oplus E_8(-1)^2\oplus \ZZ e_2\cong
(-2)\oplus (2)\oplus U\oplus E_8(-1)^2\oplus (-2).
\end{equation}
 We have $\Omega_{\Gamma}=e_3^{\bot}\cap\Omega_{\Lambda}$. Viewing $\wt{O}(\Gamma)$ as the subgroup of $\wt{O}(\Lambda)$ fixing $e_3$ we get a natural map
 \begin{equation}\label{normesse}
\nu\colon\DD^{BB}_{\Gamma}\lra \ov{\mathbb S}_2^{\star}.
\end{equation}
\begin{clm}\label{clm:normesse}
Map~\eqref{normesse} is the normalization of ${\mathbb S}_2^{\star}$.
\end{clm}
\begin{proof}
Since $\DD_{\Gamma}^{BB}$ is normal and $\nu$ is finite it sufffices to show that $\nu$ has degree $1$. Let $[\sigma]\in e_3^{\bot}\cap\Omega_{\Lambda}$ be generic. Let $g\in\wt{O}(\Lambda)$ and $[\tau]=g([\sigma])$. We must show that there exists $g'\in\wt{O}(\Gamma)$ such that $[\tau]=g'([\sigma])$. Since $[\sigma]$ is generic  we have that
\begin{equation}
\sigma^{\bot}\cap \Lambda=\ZZ e_3.
\end{equation}
It follows that $g(e_3)=\pm e_3$. If $g(e_3)= e_3$ then $g\in \wt{O}(\Gamma)$ and we are done. Suppose that $g(e_3)=- e_3$. Let $g':=(-1_{\Lambda})\circ g$. Since  multiplication by $2$  kills $D(\Lambda)$ we have that $(-1_{\Lambda})\in\wt{O}(\Lambda)$ and hence $g'\in \wt{O}(\Lambda)$: in fact  $g'\in \wt{O}(\Gamma)$ because $g'(e_3)=e_3$.  On the other hand $[\tau]=g'([\sigma])$ because $(-1_{\Lambda})$ acts trivially on $\Omega_{\Lambda}$.
\end{proof}
Our next task will be to define a finite map from $\DD_{\Gamma}^{BB}$ to the Baily-Borel compactification of the period space for $K3$ surfaces with a polarization of degree $2$. Let 
\begin{equation}
\wt{\Phi}:=J\oplus \la v_3,(v_3+e_2)/2 \ra \oplus N\oplus E_8(-1)^2 \cong
U^3\oplus  E_8(-1)^2.
\end{equation}
 Then $\wt{\Phi}$ is isometric to the \emph{$K3$ lattice} i.e.~$H^2(K3;\ZZ)$ equipped with the intersection form. 
 Let 
\begin{equation}\label{retdue}
\Phi:=v_1^{\bot}\cap \wt{\Phi}:=\ZZ e_1 \oplus\la v_3,(v_3+e_2)/2 \ra \oplus N\oplus E_8(-1)^2\oplus \cong
(-2)\oplus U^2\oplus  E_8(-1)^2.
\end{equation}
Then $\DD_\Phi$ is the period space for $K3$ surfaces with a polarization of degree $2$. 
\begin{clm}\label{clm:unisov}
$\wt{\Phi}$ is the unique   lattice contained in $\wt{\Lambda}_\QQ$ (with quadratic form equal to the  restriction of the quadratic form on $\wt{\Lambda}_\QQ$) and containing $\wt{\Gamma}$ as a  sublattice of index $2$. 
\end{clm}
\begin{proof}
First it is clear that $\wt{\Gamma}$ is contained in $\wt{\Phi}$  as a  sublattice of index $2$. 
Now suppose that $L$ is  a  lattice contained in $\wt{\Lambda}_\QQ$ and containing $\wt{\Gamma}$ as a  sublattice of index $2$. Then $L$ must be generated by $\wt{\Gamma}$ and an isotropic element of $D(\wt{\Gamma})$: since there is a unique such element $L$ is unique.
\end{proof}
By~\Ref{clm}{unisov} every isometry of $\wt{\Lambda}$ induces an isometry of $\wt{\Phi}$. It follows that 
we have well-defined injection $\wt{O}(\Lambda)<\wt{O}(\Phi)$. Since $\Omega_\Lambda=\Omega_\Phi$  
there is an induced finite map
\begin{equation}\label{mapparo}
\rho\colon \DD^{BB}_\Gamma\lra  \DD^{BB}_\Phi.
\end{equation}
\begin{rmk}\label{rmk:gradoro}
Keep notation as above. Then $\deg\rho=2^{20}-1$. 
\end{rmk}
 \subsection{Determinant of a variable quadratic form}\label{subsec:formequadr}
 \setcounter{equation}{0}
Let $U$ be a complex vector-space of finite dimension  $d$. We view $\Sym^2 U^{\vee}$ as the vector-space of quadratic forms on $U$.    Given $q\in \Sym^2 U^{\vee}$ we let $\wt{q}\colon U\to U^{\vee}$  be the associated  symmetric map. Let $K:=\ker q$; then $\wt{q}$ may be viewed as  a (symmetric)  map $\wt{q}\colon (U/K)\to \Ann K$.
The {\it dual} quadratic form $q^{\vee}$ is the quadratic form associated to the symmetric map $\wt{q}^{-1}\colon \Ann K\to (U/K)$; thus $q^{\vee}\in \Sym^2 (U/K)$. We will denote by $\wedge^i q$ the quadratic form induced by $q$ on $\bigwedge^i U$. 
\begin{rmk}\label{rmk:quadrest}
If $0\not=\alpha=v_1\wedge\ldots\wedge v_i$ is a decomposable vector of $\bigwedge^i U$ then $\wedge^i q(\alpha)$ is equal to the determinant of the Gram matrix of $q|_{\la v_1,\ldots, v_i\ra}$ with respect to the basis $\{v_1,\ldots, v_i\}$. 
\end{rmk}
The following exercise in linear algebra will be handy.
\begin{lmm}\label{lmm:primosem}
Suppose that $q\in \Sym^2 U^{\vee}$ is  non-degenerate.   Let $S\subset U$ be a subspace. Then
 \begin{equation}
\cork(q|_S)=\cork(q^{\vee}|_{\Ann  (S)}).
\end{equation}
\end{lmm}
Let $q_{*}\in \Sym^2 U^{\vee}$.  Then
\begin{equation}\label{granfi}
\det(q_{*}+q)=\Phi_0(q)+\Phi_1(q)+\ldots +\Phi_d(q),
\qquad \Phi_i\in \Sym^i (\Sym^2 U).
\end{equation}
 Of course $\det(q_{*}+q)$ is well-defined up to multiplication  by a non-zero scalar and hence so are the $\Phi_i$'s. 
The result below is  well-known (it follows from a straightforward  computation).
\begin{prp}\label{prp:conodegenere}
Let $q_{*}\in\Sym^2 U^{\vee}$
 and 
\begin{equation}
K:=\ker(q_{*}), \qquad k:=\dim K. 
\end{equation}
  Let $\Phi_i$ be the polynomials appearing in~\eqref{granfi}. Then
\begin{itemize}
\item[(1)]
$\Phi_i=0$ for $i<k$, and
\item[(2)]
there exists $c\not=0$ such that $\Phi_k(q)=c\det(q|_K)$. 
\end{itemize}
\end{prp}
Keep notation and hypotheses as in~\Ref{prp}{conodegenere}.
Let $\cV_K\subset \Sym^2 U^{\vee}$ be the subspace of quadratic forms whose restriction to $K$ vanishes. Given $q\in\cV_K$ we have $\wt{q}(K)\subset \Ann K$ and hence it makes sense to consider the restriction of $q_{*}^{\vee}$ to $\wt{q}(K)$. 
\begin{prp}\label{prp:zeronucleo}
Keep notation and hypotheses as in~\Ref{prp}{conodegenere}.
 The restriction of $\Phi_i$ to $\cV_K$ vanishes for $I<2k$. 
Moreover there exists $c\not=0$ such that
\begin{equation}
\Phi_{2k}(q)=c\det(q_{*}^{\vee}|_{\wt{q}(K)}),\qquad q\in\cV_K.
\end{equation}
\end{prp}
\begin{proof}
Choose a basis $\{u_1,\ldots,u_d\}$ of $U$ such that $K=\la u_1,\ldots,u_k\ra$ and $\wt{q}_{*}(u_i)=u_i^{\vee}$ for $k<i\le d$. Let $M$ be the Gram matrix of $q$ in the chosen basis. Expanding $\det(q_{*}+tq)$ we get that
\begin{equation*}
\det(q_{*}+tq)\equiv (-1)^k t^{2k}\sum_{J}(\det M_{{\bf k},J})^2\pmod{t^{2k+1}}
\end{equation*}
where $M_{{\bf k},J}$ is the $k\times k$ submatrix of $M$ determined by the first $k$ rows and the columns  indicized by $J=(j_1,j_2,\ldots, j_k)$. The claim follows from the equality 
\begin{equation*}
\sum_{J}(\det M_{{\bf k},J})^2=\wedge^k (q_{*}^{\vee})(\wt{q}(u_1)\wedge\ldots\wt{q}(u_k))
\end{equation*}
and~\Ref{rmk}{quadrest}.
\end{proof}
Now suppose that 
\begin{equation}\label{corango}
\cork \wt{q}_{*}=1,\qquad \ker \wt{q}_{*}=\la e\ra.
\end{equation}
We let
$\ov{q}_{*}\in\Sym^2(U/\la e\ra)^{\vee}$  be the non-degenerate quadratic form induced by $q_{*}$ i.e.~
$\ov{q}_{*}(\ov{v}):=q_{*}(v)$ for $\ov{v}\in U/\la e\ra$. Let $...,\Phi_i,...$ be as in~\eqref{granfi}. In particular $\Phi_0=0$.  Assume that
 \begin{equation}\label{eccoelle}
 L\subset\ker\Phi_1=\{q\mid q(e)=0\} 
\end{equation}
  is a vector subspace. Thus
\begin{equation}
\det(q_{*}+q)|_{L}=\Phi_2|_{L}+\ldots +\Phi_d|_{L}.
\end{equation}
We will compute $\rk(\Phi_2|_{L})$. 
Let $T\subset U$ be 
defined by
\begin{equation}\label{eccoti}
T:=\Ann  \la \wt{q}(e)\ra_{q\in L }
\end{equation}
where $L$ and $e$ are as above.
Geometrically: $\PP(T)$ is the projective tangent space at $[e]$ of the intersection of the projective quadrics parametrized by $\PP(L)$.   
\begin{prp}\label{prp:beppegrillo}
Suppose that $L\subset \Sym^2 U^{\vee}$ is a vector subspace such that~\eqref{eccoelle} holds. Keep notation as above, in particular $T$ is given by~\eqref{eccoti}. Then
\begin{equation}\label{grillino}
\rk (\Phi_2|{_L})=\cod(T,U)-\cork(\ov{q}_{*}|_{T/\la e\ra}).
\end{equation}
(The last term on the right-side makes sense because $T\supset\la e\ra$.)
\end{prp}
\begin{proof}
Let 
\begin{equation}
\begin{matrix}
L & \overset{\alpha}{\lra} & (U/\la e\ra )^{\vee} \\
q & \mapsto & \wt{q}(e).
\end{matrix}
\end{equation}
By~\Ref{prp}{zeronucleo} we have
\begin{equation}
\rk (\Phi_2|_{L})=\rk (\ov{q}^{\vee}_{*}|_{\im(\alpha)}).
\end{equation}
On the other hand~\Ref{lmm}{primosem} gives that
\begin{equation}
\rk (\ov{q}_{*}^{\vee}|_{\im(\alpha)})=\dim \im(\alpha)-
\cork(\ov{q}_{*}|_{\Ann  (\im(\alpha))}).
\end{equation}
By definition $\Ann  (\im(\alpha))=T/\la e\ra$. Since $\dim \im(\alpha)=\cod(T, U)$ we get the proposition.
\end{proof}
 \section{First extension of the period map}\label{sec:primest}
 \setcounter{equation}{0}
\subsection{Local structure of $Y_A$ along a singular plane}
\setcounter{equation}{0}
Let $(W,A)\in\wt{\Sigma}$. Then $\PP(W)\subset Y_A$. In this section we will analyze the local structure of $Y_A$ at  $v_0 \in(\PP(W)\setminus C_{W,A})$ under mild hypotheses on $A$. Let $[v_0 ]\in\PP(W)$ - for the moment being we do not require that $v_0 \notin C_{W,A}$. Let $V_0\subset V$ be a subspace transversal to $[v_0 ]$. We identify $V_0$ with an open affine neighborhood of $[v_0 ]$ via~\eqref{apertoaffine}; thus $0\in V_0$ corresponds to $[v_0 ]$. Let $f_i\in \Sym^i V_0^{\vee}$ for $i=0,\ldots,6$ be the polynomials  of~\eqref{taylor}. 
By Item~(2) of~\Ref{prp}{ipsing} we have
\begin{equation}\label{espansione}
Y_A\cap V_0=V(f_2+\ldots+f_6).
\end{equation}
Suppose that $Y_A\not=\PP(V)$. Then $[v_0 ]\in \sing Y_A$; 
since $[v_0 ]$ is an arbitrary point of $\PP(W)$  we get that $\PP(W)\subset \sing Y_A$.
It follows that  $\rk f_2\le 3$.
\begin{prp}\label{prp:germediy}
Let $(W,A)\in\wt{\Sigma}$ and suppose that $Y_{\delta(A)}\not=\PP(V^{\vee})$. Let  $[v_0 ]\in(\PP(W)\setminus C_{W,A})$ and $f_2$ be the quadratic term of the Taylor expansion of a local equation of $Y_A$ centered at $[v_0]$.   Then
\begin{equation}\label{rinogaetano}
\rk f_2=4-\dim(A\cap (\bigwedge^2 W\wedge V)).
\end{equation}
\end{prp}
\begin{proof}
By hypothesis   there exists a subspace $V_0\subset V$ such that~\eqref{trasuno}-\eqref{trasdue} hold. 
Let  $\wt{q}_A $ be as in~\eqref{aprile} and $q_A$  be the associated quadratic form on $\bigwedge^2 V_0$. 
 Let $Q_A:=V(q_A)\subset\PP(\bigwedge^ 2 V_0)$.
By~\Ref{prp}{ipsdisc} we have
\begin{equation}
V(Y_A)|_{V_0}=V(\det(q_A+q_v))
\end{equation}
where $q_v$ is as in~\eqref{quadpluck}. 
Let $W_0:=W\cap V_0$.   Since $[v_0 ]\notin C_{W,A}$ we have $A\cap (v_0\wedge\bigwedge^2 V)=\bigwedge^ 3 W$. By~\eqref{nucqua} we get that 
$\sing Q_A=\{[\bigwedge^ 2 W_0]\}$.  Thus 
\begin{equation}
\det(q_A+q_v)=\Phi_2(v)+\Phi_3(v)+\ldots+\Phi_6(v),\qquad \Phi_i\in \Sym^i V_0^{\vee}
\end{equation}
and  the rank of $\Phi_2$ is given by~\eqref{grillino}  with $q_{*}=q_A$ and $L=V_0$.
Let us identify the subspace $T\subset\bigwedge^ 2 V_0$ given by~\eqref{eccoti}. Let $U_0\subset V_0$ be a subspace transversal to $W_0$; since the Pl\"ucker quadrics generate the ideal of the Grassmannian we have 
\begin{equation}
T=\bigwedge^ 2 W_0\oplus W_0\wedge U_0.
\end{equation}
By~\Ref{prp}{beppegrillo} we get that
 \begin{equation}\label{mango}
\rk\Phi_2=3-\dim\ker(q_A|_{W_0\wedge U}).
\end{equation}
We claim that
\begin{equation}\label{ucraina}
\dim\ker(q_A|_{W_0\wedge U}) =\dim(A\cap (\bigwedge^2 W\wedge V)).
\end{equation}
In fact let $\alpha\in W_0\wedge U$. Then $\alpha\in\ker(q_A|_{W_0\wedge U})$ if and only if 
\begin{equation}
\wt{q}_A(\alpha)\in \Ann(W_0\wedge U_0)=
\bigwedge^ 2 W_0\wedge U_0\oplus \bigwedge^ 3 U_0.
\end{equation}
 Since $A\subset(\bigwedge^ 3 W)^{\bot}$ it follows from~\eqref{giannibrera}  that necessarily  
$\wt{q}_A(\alpha)\in \bigwedge^ 2 W_0\wedge U_0$.
 Equation~\eqref{giannibrera}  gives a linear map
\begin{equation}\label{cazzare}
\begin{matrix}
\ker(q_A|_{W_0\wedge U_0}) & \overset{\varphi}\lra & A\cap (\bigwedge^ 2 W\wedge U_0)\\
\alpha & \mapsto & v_0 \wedge\alpha+\wt{q}_A(\alpha).
\end{matrix}
\end{equation}
The direct-sum decomposition
\begin{equation}\label{dikdik}
\bigwedge^ 2 W\wedge U_0=[v_0 ]\wedge W_0\wedge U_0\oplus\bigwedge^ 2 W_0\wedge U_0
\end{equation}
shows that $\varphi$ is bijective. Since there is an obvious isomorphism $(A\cap (\bigwedge^ 2 W\wedge U_0))\cong (A\cap (\bigwedge^2 W\wedge V))/\bigwedge^3 W$ we get that~\eqref{ucraina} holds.  
 \end{proof}
\begin{rmk}
Suppose that $\dim(A\cap (\bigwedge^2 W\wedge V))>4$. Then  Equation~\eqref{rinogaetano} does not make sense. On the other hand  $C_{W,A}=\PP(W)$ by~\Ref{prp}{cnesinerre} and hence there is no $[v_0 ]\in(\PP(W)\setminus C_{W,A})$.
\end{rmk}
\subsection{The extension}
\setcounter{equation}{0}
\begin{lmm}\label{lmm:platone}
Let $A_0\in(\lagr\setminus\Sigma_{\infty}\setminus\Sigma[2])$.  Then $Y_{A_0}[3]$ is finite and $C_{W,A_0}$ is a sextic curve  for every $W\in\Theta_{A_0}$. 
\end{lmm}
\begin{proof}
$Y_{A_0}[3]$ is finite by~\Ref{prp}{delinf}. Let $W\in\Theta_{A_0}$. 
Let us show that $\cB(W,A_0)\ne\PP(W)$. Let $W'\in(\Theta_{A_0}\setminus\{W\})$. Then $\dim(W\cap W')=1$ because otherwise $\bigwedge^3 W$ and $\bigwedge^3 W'$ span a line in $\Gr(3,V)$ which is contained in $\Theta_{A_0}$ and that  contradicts the assumption that $\Theta_{A_0}$ is finite. This proves finiteness of the set of $[w]\in\PP(W)$ such that Item~(1) of~\Ref{dfn}{malvagio} holds. Since $\dim (\bigwedge^2 W\wedge V)\le 2$ a similar argument gives finiteness  of  the set of $[w]\in\PP(W)$ such that Item~(2) of~\Ref{dfn}{malvagio} holds. This proves that $\cB(W,A_0)$ is finite, in particular 
$\cB(W,A_0)\ne\PP(W)$.   By~\Ref{prp}{cnesinerre} it follows that $C_{W,A_0}$ is a sextic curve.
\end{proof}
\begin{prp}\label{prp:platone}
Let $A_0\in(\lagr\setminus\Sigma_{\infty}\setminus\Sigma[2])$ and  ${\bf L}\subset\PP(V)$ be a generic $3$-dimensional linear subspace. 
Let   $\cU\subset(\lagr\setminus\Sigma_{\infty}\setminus\Sigma[2])$ be a sufficiently small  open set  containing $A_0$. Let  $A\in\cU$. Then the following hold:   
\begin{itemize}
\item[(a)]
The scheme-theoretic inverse image $f_{A}^{-1}{\bf L}$ is a reduced surface  with DuVal singularities.
\item[(b)]
If in addition $A_0\in \lagr^0$ then $f_{A}^{-1}{\bf L}$ is a smooth surface. 
\end{itemize}
\end{prp}
\begin{proof}
Let ${\bf L}\subset\PP(V)$ be a generic $3$-dimensional linear subspace. Then
\begin{itemize}
\item[(1)]
${\bf L}\cap Y_{A_0}[3]=\es$.
\item[(2)]
${\bf L}\cap C_{W,A_0}=\es$ for every $W\in\Theta_{A_0}$.
\end{itemize}
In fact $Y_{A_0}[3]$ is finite by~\Ref{lmm}{platone} and hence~(1) holds.    Since $\Theta_{A_0}$ is finite and $C_{W,A_0}$ is a sextic curve for every $W\in\Theta_{A_0}$ Item~(2) holds as well.
We will prove that $f_{A_0}^{-1}{\bf L}$ is reduced with DuVal singularities and that it is smooth if  $A_0\in \lagr^0$.  The result will follow  because being smooth, reduced or having DuVal singularities is an open property. Write $\Theta_{A_0}=\{W_1,\ldots,W_d\}$.  By Item~(2) above the intersection ${\bf L}\cap\PP(W_i)$ is a single point $p_i$ for $i=1,\ldots,d$. Since  $p_i\notin C_{W_i,A_0}$  the points $p_1,\ldots,p_d$ are pairwise distinct. 
By~\Ref{prp}{ipsing} we know that away 
from  $\bigcup_{W\in\Theta_{A_0}}\PP(W)$ the locally closed sets $Y_A(1)$ and $Y_A(2)$ are  smooth of dimensions $4$ and $2$ respectively. By Item~(1) it follows that $f_{A_0}^{-1}{\bf L}$ is smooth away from 
\begin{equation}\label{bruttipunti}
f_{A_0}^{-1}\{p_1,\ldots,p_d\}.
\end{equation}
It remains to show that $f_{A_0}^{-1}{\bf L}$ is DuVal at each point of~\eqref{bruttipunti}.
 Since $p_i\in Y_{A_0}(1)$ the map  $f_{A_0}$ is \'etale of degree $2$ over $p_i$, see~\eqref{nonram}. Thus $f_{A_0}^{-1}(p_i)=\{q_i^{+},q_i^{-}\}$ and $f_{A_0}$ defines an isomorphism between the germ $(X_{A_0},q_i^{\pm})$ (in the classical topology) and the germ $(Y_{A_0},p_i)$. By~\Ref{prp}{germediy}  we get that the tangent cone of $f_{A_0}^{-1}{\bf L}$ at $q_i^{\pm}$ is a quadric cone  of rank $2$ or $3$; it follows that $f_{A_0}^{-1}{\bf L}$ has a singularity of type $A_n$ at $q_i^{\pm}$.  
\end{proof}
\begin{prp}\label{prp:dagest}
Let  $A_0\in(\lagr\setminus\Sigma_{\infty}\setminus\Sigma[2])$. Then  $\cP$  is regular at $A_0$  and  $\cP(A_0)\in\DD_{\Lambda}$.
\end{prp}
\begin{proof}
Let $\cU$ and ${\bf L}$ be as in~\Ref{prp}{platone}. Let $U\subset\cU$ be a subset containing $A_0$, open in the classical topology and contractible.
Let $U^0:=U\cap\lagr^0$. Let $\ov{A}\in U^0$; thus $X_{\ov{A}}$ is smooth. By~\Ref{lmm}{platone} we know that $f_{A}^{-1}{\bf L}$ is a smooth surface for every $A\in U^0$. Thus $\pi_1(U^0,\ov{A})$ acts by monodromy on $H^2(f_{\ov{A}}^{-1}{\bf L})$ and   by Item~(a) of~\Ref{prp}{platone} the   image of the monodromy representation is a finite group. On the other hand  $H_{\ov{A}}$ is an ample divisor on $X_{\ov{A}}$:  by the Lefschetz Hyperplane Theorem the homomorphism 
\begin{equation}\label{resmap}
H^2(X_{\ov{A}};\ZZ)\lra H^2(f_{\ov{A}}^{-1}{\bf L};\ZZ)
\end{equation}
is injective. 
The image of~\eqref{resmap} is a subgroup of $H^2(f_{\ov{A}}^{-1}{\bf L})$ invariant under the monodromy action of $\pi_1(U^0,\ov{A})$.  By injectivity of~\eqref{resmap}   the monodromy action of $\pi_1(U^0,\ov{A})$ on $H^2(X_{\ov{A}})$ is finite as well. By  Griffith's Removable Singularity Theorem (see p.~41 of~\cite{grif}) it follows that the restriction of $\cP^0$ to $U^0$ extends to a holomorphic map $U\to\DD_{\Lambda}$ .  Hence  $\cP^0$ extends regularly in a neighborhood $A_0$ and it goes into $\DD_{\Lambda}$. 
\end{proof}
\begin{dfn}\label{dfn:cappello}
Let 
$\lagrhat\subset\lagr\times\DD_{\Lambda}^{BB}$ be the closure of the graph of the restriction of $\cP$ to the set of its regular points and  
\begin{equation}\label{mappapi}
p\colon\lagrhat\to \lagr  
\end{equation}
  the restriction of projection.  Let    $\wh{\Sigma}\subset\lagrhat$ be the proper transform of $\Sigma$. 
\end{dfn}
\begin{crl}\label{crl:dagest}
Keep notation as above. Let $A$ be in the indeterminacy locus of $\cP$ and $p$ be as in~\eqref{mappapi}. Then $\dim(p^{-1}(A)\cap \wh{\Sigma})$ has dimension at least $2$.
\end{crl}
\begin{proof}
Let $\Ind(\cP)$ be the indeterminacy locus of $\cP$. Since   $\lagr$ is smooth the morphism $p$ identifies $\lagrhat$ with the blow-up of  $\Ind(\cP)$. Hence  the exceptional 
set of  $p$  is the support of a Cartier divisor $E$. 
By~\Ref{prp}{dagest}  the indeterminacy locus of $\cP$ is contained in $\Sigma$ and thus $A\in\Sigma$. It follows that  $p^{-1}(A)\cap \wh{\Sigma}$ is not empty. Since $\wh{\Sigma}$ is a prime divisor in $\lagrhat$ and $E$ is a Cartier divisor every irreducible component of $E\cap\wh{\Sigma}$ has codimension $2$ in $\lagrhat$. On the other hand~\Ref{prp}{dagest}, \eqref{romalazio} and~\eqref{codinf} give that $\cod(\Ind(\cP),\lagr)\ge 4$ and hence every component of a fiber of $E\cap\wh{\Sigma}\to\Ind(\cP)$ has dimension at least $2$. Since $p^{-1}(A)\cap \wh{\Sigma}$ is one such fiber 
we get the corollary.
\end{proof}
\section{Second extension of the period map}\label{sec:perdisig}
 \setcounter{equation}{0}
\subsection{$X_A$ for generic $A$ in $\Sigma$}
\setcounter{equation}{0}
Let $A\in(\Sigma\setminus\Sigma[2])$ and $W\in\Theta_A$. Then $\cB(W,A)\not=\PP(W)$ because by~\Ref{lmm}{platone} we know that $C_{W,A}\not=\PP(W)$. By the same Lemma $Y_A[3]$ is finite. In particular $(\PP(W)\setminus \cB(W,A)\setminus Y_A[3])$ is not empty. 
\begin{prp}\label{prp:hocuspocus}
Let $A\in(\Sigma\setminus\Sigma[2])$ and $W\in\Theta_A$. Suppose in addition that $\dim(A\cap (\bigwedge^2 W\wedge V))=1$.  Let 
\begin{equation}
x\in f_A^{-1}(\PP(W)\setminus \cB(W,A)\setminus Y_A[3]). 
\end{equation}
The germ $(X_A,x)$ of $X_A$ at $x$ in the classical topology is isomorphic to $(\CC^2,0)\times A_1$ and $\sing X_A$  is equal to $f_A^{-1}\PP(W)$ in a neighborhood of $x$.  
\end{prp}
\begin{proof}
Suppose first that $f_A(x)\notin  C_{W,A}$. Then $f_A(x)\in Y_A(1)$ and hence $f_A$ is \'etale over $f_A(x)$, 
see~\eqref{nonram}.  Thus the germ $(X_A,x)$ is isomorphic to the germ $(Y_A,f_A(x))$ and the statement of the proposition follows from~\Ref{prp}{germediy} because  by hypothesis $B=0$. It remains to examine the case 
\begin{equation}\label{gobetti}
f_A(x)\in(C_{W,A}\setminus \cB(W,A)\setminus Y_A[3]).  
\end{equation}
Let $f_A(x)=[v_0 ]$. Since $A\notin\Sigma_{\infty}$ there exists a subspace $V_0\subset V$  transversal to $[v_0 ]$ and such that~\eqref{trasdue} holds - see~\Ref{prp}{alpiudue}. Thus we may apply~\Ref{prp}{locdegsym}. We will adopt the notation of that Proposition, in particular we will identify $V_0$ with $(\PP(V)\setminus\PP(V_0))$ via~\eqref{apertoaffine}.
Let $W_0:= W\cap V_0$; thus $\dim W_0=2$. 
 Let $K\subset\bigwedge^2 V_0$ be the subspace corresponding to $(v_0\wedge\bigwedge^2 V)\cap A$   via~\eqref{iaquinta}.
By~\eqref{gobetti} $\dim K=2$.
 Let us prove that there exists a basis $\{w_1,w_2,u_1,u_2,u_3\}$ of $V_0$ such that $w_1,w_2\in W_0$ and 
\begin{equation}\label{formstan}
K=\la w_1\wedge w_2,\ w_1\wedge u_1+u_2\wedge u_3\ra.
\end{equation}
In fact since $[v_0]\notin\cB(W,A)$ the following hold:
\begin{itemize}
\item[(1)]
$\PP(K)\cap\Gr(2,V_0)=\{\bigwedge^2 W_0\}$.
\item[(2)]
$\PP(K)$ is \emph{not} tangent to $\Gr(2,V_0)$.
\end{itemize}
Now let $\{\alpha,\beta\}$ be a basis  of $K$ such that $\bigwedge^ 2 W_0=\la \alpha\ra$. By~(1) we have that $\beta\wedge\beta\not=0$. Let $S:=\supp(\beta\wedge\beta)$; thus $\dim S=4$. Let us prove that $W_0\not\subset S$. In fact suppose that $W_0\subset S$. Then $K\subset\bigwedge^ 2 S$ and since $\Gr(2,S)$ is a quadric hypersurface in $\PP(\bigwedge^ 2 S)$ it follows that either $\PP(K)$ intersects $\Gr(2,U)$  in two points or  is tangent to it, that contradicts~(1) or~(2) above.
Let $\{w_1,w_2\}$ be a basis of $W_0$ such that $w_1\in W_0\cap S$; it is clear that there exist $u_1,u_2,u_3\in S$ linearly independent such that $\beta=w_1\wedge u_1+u_2\wedge u_3$. This proves that~\eqref{formstan}  holds. 
Rescaling   $u_1,u_3$ we  may assume that
\begin{equation}\label{sturzo}
\vol_0 (\wedge w_1\wedge w_2\wedge u_1\wedge u_2\wedge u_3)=1
\end{equation}
where $\vol_0$ is our chosen volume form, see~\eqref{volzero}. Let 
\begin{equation}\label{basebase}
J:=\la w_1\wedge u_1,\,w_1\wedge u_2,\,
w_1\wedge u_3,\,w_2\wedge u_1,\,w_2\wedge u_2,\,
w_2\wedge u_3,\,u_1\wedge u_2,\ u_1\wedge u_3\ra.
\end{equation}
Thus $J$ is transversal to  $K$ by~\eqref{formstan} and hence  we have Decomposition~\eqref{donmilani}. Given $v\in V_0$ we write
\begin{equation}
v=s_1 w_1+s_2 w_2+t_1 u_1+t_2 u_2+ t_3 u_3.
\end{equation}
Thus $(s_1,s_2,t_1,t_2,t_3)$ are affine coordinates on $V_0$ and hence by~(\ref{apertoaffine}) they are also coordinates on an open neighborhood of $[v_0 ]\in V_0$.   Let $N=N_{J}$, $P=P_{J}$, $Q=Q_{J}$, $R=R_{J}$  be the matrix functions appearing in~\eqref{patrici} . A straightforward computation gives that
\begin{equation}\label{nittygritty}
P(v)=\left(
\begin{matrix}
0 & t_1 \\
t_1 & -2s_2
\end{matrix}
\right),\quad
R(v)=
\left(
\begin{matrix}
0 & 0 & 0 & 0 & 0 & 0 & t_3 & -t_2 \\
-s_2 & 0 & 0 & s_1 & -t_3 & t_2 & 0 & 0
\end{matrix}
\right).
\end{equation}
 The $8\times 8$-matrix $(N+Q(v))$ is invertible for $(s,t)$ in a neighborhood of $0$; we set
 \begin{equation}
(c_{ij})_{1\le i,j\le 8}=-(N+Q(v))^{-1}
\end{equation}
where $c_{ij}\in\cO_{V_0,0}$. 
A straightforward computation gives that
\begin{equation}\label{ecchilo}
P(v)-R(v)\cdot(N+Q(v))^{-1}\cdot R(v)^t= 
\left(
\begin{matrix}
c_{7,7}t_3^2-2 c_{7,8}t_2 t_3+c_{8,8} t_2^2 & 
t_1+\delta \\
t_1+\delta  & -2s_2+\epsilon
\end{matrix}
\right)
\end{equation}
where $\delta,\epsilon\in{\mathfrak m}^2_0$ (here ${\mathfrak m}_0\subset\CC[s_1,s_2,t_1,t_2,t_3]$ is the maximal ideal of $(0,\ldots,0)$).
Let us prove that 
\begin{equation}\label{tortora}
\det
\left(
\begin{matrix}
c_{7,7}(0) & -c_{7,8}(0) \\
-c_{8,7}(0) & c_{8,8}(0)
\end{matrix}
\right)\not=0.
\end{equation}
Since $Q(0)=0$ we have $c_{ij}(0)=-(\det N)^{-1}\cdot N^{ij}$ where $N^c=(N^{ij})_{1\le i,j\le 8}$ is the matrix of cofactors of $N$. Thus~(\ref{tortora}) is equivalent to  
\begin{equation}\label{ciccio}
\det
\left(
\begin{matrix}
N^{7,7} & N^{7,8} \\
N^{8,7} & N^{8,8}
\end{matrix}
\right)\not=0.
\end{equation}
The quadratic form $q_A|_{J}$ is non-degenerate and hence we have the dual quadratic form $(q_A|_{J})^{\vee}$  on $J^{\vee}$.  Let  $U:=\la u_1,u_2,u_3\ra$ where  $u_1,u_2,u_3$ are as in~\eqref{formstan}. Applying~\Ref{lmm}{primosem} to $q_A|_{J}$ and the subspace  $W_0\wedge U\subset J$ we get that 
\begin{equation}
\cork(q_A|_{W_0\wedge U})=\cork((q_A|_{J})^{\vee}|_{\Ann  (W_0\wedge U)}).
\end{equation}
By~\eqref{ucraina} $q_A|_{W_0\wedge U}$ is non-degenerate: it follows that $(q_A|_{J})^{\vee}|_{\Ann  (W_0\wedge U)}$ is non-degenerate as well. The annihilator of $W_0\wedge U$ in $J^{\vee}$ is given
by
\begin{equation}\label{pappagone}
\Ann  (W_0\wedge U)= \la u_1^{\vee}\wedge u_2^{\vee},\ u^{\vee}_1\wedge u_3^{\vee}\ra
\end{equation}
 and  the Gram-matrix of $(q_A|_{J})^{\vee}|_{\Ann  (W_0\wedge U)}$ with respect to the  basis given by~\eqref{pappagone} is equal to $(\det N)^{-1}(N^{ij})_{7\le i,j\le 8}$. Hence~\eqref{ciccio} holds and this proves that~\eqref{tortora} holds. By~(\ref{ecchilo}) and~(\ref{tortora}) there exist new analytic coordinates  $(x_1,x_2,y_1,y_2,y_3)$ on an open neighborhood $\cU$ of $0\in V_0$  - with $(0,\ldots,0)$ corresponding to $0\in V_0$ - such that
\begin{equation}\label{matricedue}
P(v)-R(v)\cdot (N+Q(v))^{-1}\cdot R(v)^t= 
\left(
\begin{matrix}
x_1^2+x_2^2 &  y_1 \\
y_1 & y_2
\end{matrix}
\right).
\end{equation}
(Recall that $\delta,\epsilon\in{\mathfrak m}^2_0$.) By~\Ref{prp}{locdegsym}  we get that
\begin{equation}\label{benni}
f_A^{-1}\cU=V(\xi_1^2-y_2,\ \xi_1\xi_2+y_1,\ \xi_2^2-x_1^2-x_2^2,)\subset\cU\times\CC^2
\end{equation}
where $(\xi_1,\xi_2)$ are coordinates on $\CC^2$ and our point $x\in X_A$ has  coordinates $(0,\ldots,0)$. 
(Notice that if $k=2$ then the entries of the first matrix of~\eqref{eqespl} belong to the ideal generated by the entries of the second matrix of~\eqref{eqespl}.)
Let $B^3(0,r)\subset\CC^3$ be a small open ball centered at the origin and let $(x_1,x_2,y_3)$ be coordinates on $\CC^3$; there is an obvious isomorphism between an open neighborhood of $0\in f_A^{-1}\cU$ and
\begin{equation}\label{barsport}
V(\xi_2^2-x_1^2-x_2^2)\subset B^3(0,r)\times\CC^2
\end{equation}
taking $(0,\ldots,0)$ to $(0,\ldots,0)$. 
This proves that $X_A$ is singular at $x$ with analytic germ as claimed. It follows that $f_A^{-1}(\PP(W)\setminus\cB(W,A)\setminus Y_A[3])\subset \sing Y_A$. On the other hand an arbitrary point $x'$ in a sufficiently small neighborhood of $x$ is mapped to $Y_A(1)$ and if it does not belong to $f_A^{-1}\PP(W)$ the map $f_A$ is \'etale at $x'$: by~\Ref{prp}{ipsing} $Y_A$ is smooth at $f(x')$  and therefore $X_A$ is smooth at $x'$.
\end{proof}
Let $\Sigma^{\rm sm}$ be the smooth locus of $\Sigma$.    
\begin{crl}\label{crl:runrabbit}
Let $A\in(\Sigma^{\rm sm}\setminus\Delta)$ and $W$ be the unique element in $\Theta_A$ (unique by~\eqref{singsig}). Then
\begin{itemize}
\item[(1)]
$\sing X_A=f_A^{-1}\PP(W)$.
\item[(2)]
 Let $x\in f_A^{-1}\PP(W)$. The germ $(X_A,x)$  in the classical topology is isomorphic to $(\CC^2,0)\times A_1$.
\item[(3)]
$C_{W,A}$ is a smooth sextic curve in $\PP(W)$. 
\item[(4)]
The map
\begin{equation}\label{doppiopasso}
\begin{matrix}
f_A^{-1}\PP(W) & \lra & \PP(W) \\
x & \mapsto & f_A(x)
\end{matrix}
\end{equation}
is a double cover simply branched over $C_{W,A}$. 
\end{itemize}
\end{crl}
\begin{proof}
(1)-(2): By~\eqref{singsig} $A\notin (\Sigma_{\infty}\cup\Sigma[2])$, $\dim(A\cap (\bigwedge^2 W\wedge V))=1$ and $\cB(W,A)=\es$. Moreover  $Y_A[3]$ is empty by definition.  
By~\Ref{prp}{hocuspocus} it follows that   $f_A^{-1}\PP(W)\subset\sing X_A$ and that the analytic germ at $x\in f_A^{-1}\PP(W)$ is as stated. It remains to prove that $X_A$ is smooth at $x\in(X_A\setminus f_A^{-1}\PP(W))$. Since $A\notin\Delta$ we have that $f_A(x)\in(Y_A(1)\cup Y_A(2))$. If $f_A(x)\in Y_A(1)$ then $f_A$ is \'etale over $f_A(x)$ (see~\eqref{nonram}) and $Y_A$ is smooth at $f_A(x)$ by~\Ref{prp}{ipsing}: it follows that $X_A$ is smooth at $x$.  If $f_A(x)\in Y_A(2)$ then $X_A$ is smooth at $x$ by Lemma~2.5 of~\cite{ogdoppio}.
(3): Immediate consequence of~\Ref{prp}{cnesinerre}. 
(4):  Map~\eqref{doppiopasso} is an \'etale cover away from $C_{W,A}$, see~\eqref{nonram},  while 
 $f_A^{-1}(y)$ is a single point for $y\in C_{W,A}$ - see~\eqref{benni}. Thus either $f_A^{-1}\PP(W)$ is singular or else Map~\eqref{doppiopasso} is simply branched over $C_{W,A}$. 
Items~(1), (2) show that  $f_A^{-1}\PP(W)$ is smooth: it follows that Item~(4) holds. 
\end{proof}
\begin{dfn}
Suppose that $(W,A)\in\wt{\Sigma}$ and that $C_{W,A}\not=\PP(W)$. We let
\begin{equation}
S_{W,A}\lra\PP(W)
\end{equation}
be the double cover  ramified over $C_{W,A}$. If  $\Theta_A$ has a single element we let $S_A:=S_{W,A}$.
\end{dfn}
\begin{rmk}
Let $A\in(\Sigma^{\rm sm}\setminus\Delta)$ and $W$ be the unique element of  $\Theta_A$. By Item~(4) of~\Ref{crl}{runrabbit} $f_A^{-1}\PP(W)$ is identified with $S_A$ and the restriction of $f_A$ to $f_A^{-1}\PP(W)$ is identified with the double cover $S_A\to\PP(W)$. In particular
$f_A^{-1}\PP(W)$ is a polarized $K3$ surface of degree $2$.
\end{rmk}
\subsection{Desingularization of $X_A$ for $A\in(\Sigma^{\rm sm}\setminus\Delta)$}\label{subsec:singper}
\setcounter{equation}{0}
Let $A\in(\Sigma^{\rm sm}\setminus\Delta)$ and $W$ be the unique element of $\Theta_A$. 
Let 
\begin{equation}
\pi_{A} \colon \wt{X}_{A} \to X_{A} 
\end{equation}
 be the blow-up of $\sing X_{A} $. Then $\wt{X}_{A} $ is smooth by~\Ref{crl}{runrabbit}. Let 
\begin{equation}\label{accatilde}
\wt{H}_{A}:=\pi_{A}^{*}H_{A} ,\quad 
\wt{h}_{A}:=c_1(\cO_{\wt{X}_A }(\wt{H}_A)).
\end{equation}
Let  
\begin{equation}\label{cricket}
\cU\subset(\lagr\setminus\sing\Sigma\setminus\Delta)
\end{equation}
 be an open (classical topology) contractible neighborhood of $A$. 
We may assume  that there exists a tautological family of double EPW-sextics $\cX\to\cU$, see \S 2 of~\cite{ogdoppio}.  Let $\cH $ be the tautological divisor class on $\cX $: thus  $\cH |_{X_A }\sim H_A $.
 The holomorphic line-bundle $\cO_{\cU}(\Sigma)$ is trivial and hence there is a well-defined double cover $\phi\colon\cV\to\cU$ ramified over $\Sigma\cap\cU$.  Let $\cX_2 :=\cV\times_{\cU}\cX $ be the base change:
\begin{equation}\label{cambiobase}
\xymatrix{
\cX_2   \ar^{\wt{\phi}}[r]  \ar^{\rho_2}[d] & \cX  \ar^{\rho}[d] \\
\cV  \ar^{\phi}[r] & \cU.}
\end{equation}
Given $A'\in\Sigma\cap\cU$ we will  denote by the same symbol   the unique point in $\cV$ lying over $A'$.
\begin{prp}\label{prp:simres}
Keep notation and assumptions as above.  There is a simultaneous resolution of singularities $\pi\colon \wt{\cX}\to\cX$ fitting into a commutative diagram
\begin{equation}\label{triangolo}
\xymatrix{
\wt{\cX}   \ar^{\pi}[rr]  \ar^{g}[dr] &    & \cX_2    \ar^{\rho_2}[dl] \\
 & \cV &}
\end{equation}
Moreover $\pi$ is an isomorphism away from $g^{-1}(\phi^{-1}(\Sigma\cap \cU))$ and 
\begin{equation}\label{rumba}
g^{-1}(A)  \cong\wt{X}_{A}, \qquad \pi|_{g^{-1}(A)}  =\pi_{A},\qquad 
\pi^{*}\cH |_{g^{-1}(A)}  \sim \wt{H}_{A}. 
\end{equation}
\end{prp}
\begin{proof}
By Proposition~3.2 of~\cite{ogdoppio} $\cX$ is smooth and the map $\rho$ of~\eqref{cambiobase} is a submersion of smooth manifolds away from points $x\in\cX$ such that 
\begin{equation}\label{condax}
\rho(x):=A'\in\Sigma\cap\cU,\quad x\in S_{A'}.
\end{equation}
Let $(A',x)$ be as in~\eqref{condax}. By~\Ref{prp}{hocuspocus} and smoothness of $\cX$ we get that the map of analytic germs $(\cX,x)\to (\cU,A')$ is isomorphic to 
\begin{equation}
\begin{matrix}
(\CC^3_{\xi}\times\CC^2_{\eta}\times\CC^{53}_t,{\bf 0}) & \lra & (\CC^{54}_t,{\bf 0}) \\
(\xi,\eta,t) & \mapsto & (\xi_1^2+\xi_2^2+\xi_3^2,t_2,\ldots,t_{54})
\end{matrix}
\end{equation}
Thus~\eqref{triangolo} is obtained by the classical process of simultaneous resolution of ordinary double points of surfaces.
More precisely let $\wh{\cX}_2 \to\cX_2 $ be the blow-up of $\sing\cX_2 $. Then $\wh{\cX}_2 $ is smooth and the exceptional divisor is a fibration over $\sing\cX_2 $ with fibers isomorphic to $\PP^1\times\PP^1$. Since $\sing\cX_2 $ is simply-connected we get that the exceptional divisor has two rulings by $\PP^1$'s. It follows that there are two small resolutions of $\cX_2 $ obtained by  contracting the exceptional divisor along either one of the two rulings.  Choose one small resolution and call it $\wt{\cX}_2 $.  Then~\eqref{rumba} holds.
\end{proof}
\begin{crl}\label{crl:conichesuw}
Let $A\in(\Sigma^{\rm sm}\setminus\Delta)$ and $A'\in\lagr^0$. Then  $(\wt{X}_{A},\wt{H}_{A})$ is a $HK$ variety deformation equivalent to  $(X_{A'},H_{A'})$. Moreover
$\cP(A)=\Pi(\wt{X}_A ,\wt{H}_A)$ where $\Pi(\wt{X}_A ,\wt{H}_A)$ is given by~\eqref{eccoperi}. 
\end{crl}
\begin{proof}
Since $\pi_A\colon\wt{X}_{A}\to X_A$ is a blow-up $\wt{X}_{A}$ is projective. By~\Ref{prp}{simres}  $\wt{X}_{A}$ is a (smooth) deformation of $X_{A'}$: it follows that  $\wt{X}_{A}$ is a HK variety. The remaining statements are obvious.
\end{proof}
\begin{dfn}
Let $A\in(\Sigma^{\rm sm}\setminus\Delta)$. We let  $E_A\subset \wt{X}_A $ be the exceptional divisor of $\pi_A \colon\wt{X}_A \to X_A$ and $\zeta_A:=c_1(\cO_{\wt{X}_A }(E_A))$.
\end{dfn}
Given $A\in(\Sigma^{\rm sm}\setminus\Delta)$ we have a smooth conic bundle\footnote{$p$ is a smooth map and each fiber is isomorphic to $\PP^1$.}
\begin{equation}\label{fibcon}
p\colon E_A \lra S_A.
\end{equation}
\begin{clm}
Let $(,)$ be the Beauville-Bogomolov quadratic form of $\wt{X}_A$. The following formulae hold:
\begin{eqnarray}
(\wt{h}_A,\zeta_A) & =0,\label{ortogo}\\
(\zeta_A,\zeta_A) & =-2.\label{menodue}
\end{eqnarray}
\end{clm}
\begin{proof}
We claim that
\begin{equation}\label{dopugu}
6(\zeta_A,\wt{h}_A)=\int_{\wt{X}_A }\zeta_A\wedge \wt{h}_A^3=
\int_{S_A}h_A^3=0.
\end{equation}
In fact the first equality follows from Fujiki's relation
\begin{equation}\label{relfuj}
\int_X \alpha^4=3(\alpha,\alpha)^2,\qquad \alpha\in H^2(X)
\end{equation}
valid for any deformation of the Hilbert square of a $K3$ (together with the 
fact that  $(\wt{h}_A,\wt{h}_A)=2$) and third equality in~\eqref{dopugu} holds because $\dim S_A=2$.    Equation~\eqref{ortogo} follows from~\eqref{dopugu}.  In order to prove~\eqref{menodue} we notice that  $K_{E_A}\cong \cO_E(E_A)$ by adjunction and hence 
\begin{equation}\label{duesucon}
\int_{p^{-1}(s)}\zeta_A=-2,\quad s\in S_A.
\end{equation}
Using~\eqref{relfuj}, \eqref{ortogo}  and~(\ref{duesucon}) one gets that
\begin{equation}
2(\zeta_A,\zeta_A)=(\wt{h}_A,\wt{h}_A)\cdot(\zeta_A,\zeta_A)=
\int_{\wt{X}_A }\wt{h}_A^2\wedge \zeta_A^2 =
2\int_{p^{-1}(s)}\zeta_A=-4.
\end{equation}
Equation~\eqref{menodue} follows from the above equality.
\end{proof}
\subsection{Conic bundles in HK fourfolds}\label{subsec:famconiche}
\setcounter{equation}{0}
We have shown that if $A\in(\Sigma^{\rm sm}\setminus\Delta)$ then $\wt{X}_A $ contains a divisor which is a smooth conic bundle over a $K3$ surface. In the present section we will discuss HK four-folds containing a smooth conic bundle over a $K3$ surface. (Notice that  if a divisor in a HK four-fold is a  conic bundle over a smooth base then the base is a holomorphic symplectic surface.)
\begin{prp}\label{prp:accaie}
Let $X$ be a hyperk\"ahler  $4$-fold. Suppose that $X$ contains a prime divisor $E$ which carries a conic fibration $p\colon  E \lra S$ over a $K3$ surface $S$.
Let $\zeta:=c_1(\cO_X(E))$. Then:
\begin{itemize}
\item[(1)]
$h^0(\cO_X(E))=1$ and $h^p(\cO_X(E))=0$ for $p>0$.
\item[(2)]
$q_X(\zeta)<0$ where $q_X$ is the Beauville-Bogomolov quadratic form of $X$.
\end{itemize}
\end{prp}
\begin{proof}
By adjunction $K_E\cong\cO_E(E)$ and hence 
\begin{equation}\label{trippa}
\int_{p^{-1}(s)}\zeta=-2,\quad s\in S.
\end{equation}
Thus $h^0(\cO_E(E))=0$ and hence $h^0(\cO_X(E))=1$. Let us prove that the homomorphism 
\begin{equation}\label{pinkydinky}
H^q(\cO_X)\lra H^q(\cO_E)
\end{equation}
induced by restriction is an isomorphism for $q<4$. It is an isomorphism for $q=0$ because both $X$ and $E$ are connected.
The spectral sequence with $E_2$ term $H^i(R^j(p|_E)\cO_E)$ abutting to  $H^q(\cO_E)$ gives an isomorphism $H^q(\cO_E)\cong H^q(\cO_S)$. Since $S$ is a $K3$ surface it follows that $H^q(\cO_E)=0$ for $q=1,3$. On the other hand $H^q(\cO_X)=0$ for odd $q$ because $X$ is a HK manifold. Thus~\eqref{pinkydinky} is an isomorphism for $q=1,3$. It remains to prove that~\eqref{pinkydinky} is an isomorphism for $q=2$. By Serre duality it is equivalent  to prove that  the restriction homomorphism $ H^0(\Omega^2_X)\to H^0(\Omega^2_E)$ is an isomorphism. 
Since $1= h^0(\Omega^2_X)=h^0(\Omega^2_E)$ it suffices to notice that a holomorphic symplectic form on $X$ cannot vanish on $E$ (the maximum dimension of an isotropic subspace for $\sigma|_{T_x X}$ is equal to $2$). 
This finishes the proof that~\eqref{pinkydinky}  is an isomorphism for $q<4$. 
The long exact cohomology sequence associated to 
\begin{equation}
0\lra\cO_X(-E)\lra\cO_X\lra\cO_E\lra 0
\end{equation}
gives that $h^q(\cO_X(-E))=0$ for $q<4$.  
By Serre duality we get that Item~(1) holds..
 Let $c_X$ be the Fujiki constant of $X$; thus 
\begin{equation}
\int_X\alpha^4= c_X q_X(\alpha)^2,\qquad \alpha\in H^2(X).
\end{equation}
 Let $\iota\colon E\hra X$ be Inclusion. Let $\sigma$ be a holomorphic symplectic form on $X$. We proved above that there exists a holomorphic symplectic form $\tau$ on $S$ such that $\iota^{*}\sigma=p^{*}\tau$. Thus we have
\begin{equation}
\frac{c_X}{3} q_X(\zeta) q_X(\sigma+\ov{\sigma})=
\int_X \zeta^2\wedge(\sigma+\ov{\sigma})^2=\int_E\iota^{*}\zeta\wedge p^{*}(\tau+\ov{\tau})^2=-2\int_S(\tau+\ov{\tau})^2.
\end{equation}
 (The first equality follows from $(\zeta,\sigma+\ov{\sigma})=0$, we used~(\ref{trippa}) to get the last equality.) On the other hand $c_X>0$ and $q_X(\sigma+\ov{\sigma})>0$: thus $q_X(\zeta)<0$. 
\end{proof}
Let $X$ and $E$ be as in~\Ref{prp}{accaie}. Let $\Def_E(X)\subset \Def(X)$ be the germ representing deformations for which $E$ deforms and $\Def_{\zeta}\subset \Def(X)$ be the germ representing deformations that keep $\zeta$ of type $(1,1)$. We have an inclusion of germs
\begin{equation}\label{ingerme}
\Def_E(X)\hra \Def_{\zeta}(X).
\end{equation}
\begin{crl}\label{crl:batman}
Let $X$ and $E$ be as in~\Ref{prp}{accaie}. The following hold:
\begin{itemize}
\item[(1)]
Inclusion~\eqref{ingerme} is an isomorphism.
\item[(2)]
Let $C$ be a fiber of the conic vibration $p\colon E\to S$. Then 
\begin{equation}\label{benjo}
\{\alpha\in H^2(X;\CC)\mid (\alpha,\zeta)=0\}=\{\alpha\in H^2(X;\CC)\mid \int_C\alpha=0\}.
\end{equation}
\item[(3)]
The restriction map $H^2(X;\CC)\to H^2(E;\CC)$ is an isomorphism. 
\end{itemize}
\end{crl}
\begin{proof}
Item~(1) follows at once from Item~(1) of~\Ref{prp}{accaie} and upper-semicontinuity of cohomology dimension. 
 Let us prove Item~(2).  Let $X_t$ be a very generic small deformation of $X$ parametrized by a point of $\Def_{\zeta}\subset \Def(X)$ and $\zeta_t\in H^{1,1}_{\ZZ}(X_t)$ be the class deforming $\zeta$. A non-trivial rational Hodge sub-structure of  $H^2(X_t)$ is  equal to $\zeta_t^{\bot}$ or to 
 $\CC \zeta_t$. On the other hand~\eqref{ingerme} is an isomorphism: thus $X_t$ contains a deformation $E_t$ of $E$ and hence also a deformation $C_t$ of $C$. Clearly $\{\alpha\in H^2(X_t;\CC)\mid \int_{C_t}\alpha=0\}$ is a  rational Hodge sub-structure of $H^2(X_t)$ containing $H^{2,0}(X_t)$ and non-trivial  by~\eqref{trippa}: it follows that
\begin{equation}
\{\alpha\in H^2(X_t;\CC)\mid (\alpha,\zeta_t)=0\}=\{\alpha\in H^2(X_t;\CC)\mid \int_{C_t}\alpha=0\}.
\end{equation}
The kernel of the restriction map $H^2(X_t;\CC)\to H^2(E_t;\CC)$ is a  rational Hodge sub-structure   $V_t\subset H^2(X_t)$. By~\eqref{trippa} we know that $\zeta_t\notin V_t$ and  since~\eqref{pinkydinky} is an isomorphism for $q=2$ we know that $H^{2,0}(X_t)\not\subset V_t$; thus $V_t=0$. 
Parallel transport by the Gauss-Manin connection gives Items~(2) and~(3).
\end{proof}
Let $\iota\colon E\hra X$ be Inclusion. By Items~(2) and~(3) of~\Ref{crl}{batman} we have  an isomorphism
\begin{equation}\label{roars}
\begin{matrix}
\zeta^{\bot}& \overset{\sim}{\lra} & \{\beta\in H^2(E;\CC)\mid \int_C\beta=0\}\\
& & \\
\alpha & \mapsto & \iota^{*}\alpha
\end{matrix}
\end{equation}
On the other hand $p^{*}\colon H^2(S;\CC)\to H^2(E;\CC)$ defines an isomorphism of $H^2(S;\CC)$ onto the right-hand side of~\eqref{roars}. Thus~\eqref{roars} gives an isomorphism 
\begin{equation}\label{benhur}
r\colon \zeta^{\bot} \overset{\sim}{\lra} H^2(S;\CC).
\end{equation}
\begin{clm}\label{clm:conserva}
Let $X$, $E$ be as in~\Ref{prp}{accaie} and $r$ be as in~\eqref{benhur}. Suppose in addition that  the Fujiki constant $c_X$ is equal to $3$ and that $q_X(\zeta)=-2$. 
Let $\alpha\in\zeta^{\bot}$. Then
 \begin{equation}
q_X(\alpha)=\int_S r(\alpha)^2.
\end{equation}
\end{clm}
\begin{proof}
Equality~\eqref{trippa} gives that 
\begin{equation}
-2 q_X(\alpha)=\frac{c_X}{3} q_X(\zeta) q_X(\alpha)=\int_X \zeta^2\wedge\alpha^2=\int_E\iota^{*}\zeta\wedge(\iota^{*}\alpha)^2=-2\int_S r(\alpha)^2.
\end{equation}
\end{proof}
\subsection{The period map on $(\Sigma^{\rm sm}\setminus\Delta)$}
\setcounter{equation}{0}
 Let $A_0\in(\Sigma^{\rm sm}\setminus\Delta)$. 
By~\eqref{singsig} and Cor.~2.5.1 of~\cite{ogmoduli}
 $A_0$ belongs to the GIT-stable locus of $\lagr$. 
By Luna's  \'etale slice Theorem~\cite{luna1} it follows that there exists an analytic   $PGL(V)$-slice at $A_0$, call it $Z_{A_0}$,   such that the natural map
 \begin{equation}
Z_{A_0}/\Stab(A_0)\to \gM
\end{equation}
is an isomorphism onto an open (classical topology) neighborhood of $[A_0]$. We may assume that $Z_{A_0}\subset\cU$ 
where $\cU$ is as in~\eqref{cricket}.  Let $\wt{Z}_{A_0}:=\phi^{-1}Z_{A_0}$ where $\phi\colon\cV\to\cU$ is as in~\eqref{cambiobase}. Then $\phi$ defines a double cover $\wt{Z}_{A_0}\to Z_{A_0}$ ramified over $\Sigma\cap Z_{A_0}$; if $A\in \Sigma\cap Z_{A_0}$ we will denote by the same letter  the unique point in $\phi^{-1}(A)$. By~\Ref{prp}{simres} points of $\wt{Z}_{A_0}$ parametrize deformations of $X_A$ for $A\in\lagr^0$.
Since $\Sigma$ is smooth at $A_0$ also $\Sigma\cap Z_{A_0}$ is smooth at $A_0$. Thus $\wt{Z}_{A_0}$ is smooth at $A_0$. Shrinking $Z_{A_0}$ around $A_0$ if necessary we may assume that $\wt{Z}_{A_0}$ is contractible. Hence a marking $\psi$ of $(\wt{X} _{A_0},\wt{H}_{A_0})$ defines a marking of $(\wt{X} _{A},\wt{H}_{A})$ for all $A\in Z_{A_0}$; we will denote it by 
the same letter $\psi$.
Thus we have a local period map
\begin{equation}\label{prodiere}
\begin{matrix}
\wt{Z}_{A_0} & \overset{\wt{\cP}}\lra & \Omega_{\Lambda} \\
& & \\
t & \mapsto & \psi_{\CC} (H^{2,0}(g^{-1}t)).
\end{matrix}
\end{equation}
\begin{clm}\label{clm:vendola}
The local period map $\wt{\cP}$   of~\eqref{prodiere} defines an isomorphism of a sufficiently small open neighborhood of $A_0$ in $\wt{Z}_{A_0}$ onto an open subset of $\Omega_{\Lambda}$. 
\end{clm}
\begin{proof}
Since $\wt{Z}_{A_0}$ is smooth and $\dim \wt{Z}_{A_0} =\dim \Omega_{\Lambda}$ it suffices to prove that  $d\wt{\cP}(\wt{A}_0)$ is injective. By Luna's \'etale slice Theorem we have an isomorphism of germs  
\begin{equation}\label{pescoluse}
(Z_{A_0},A_0)\overset{\sim}{\lra}\Def(X_{A_0},H_{A_0})
\end{equation}
 induced by the local tautological family of double EPW-sextics parametrized by $Z_{A_0}$.
 By~\Ref{crl}{runrabbit} the points of $Z_{A_0}\cap \Sigma$ parametrize deformations of $X_{A_0}$ which are locally trivial at points of $S_A$. Let $\wt{\Sigma}_{A_0}\subset\wt{Z}_{A_0}$ be the inverse image of $\Sigma\cap Z_{A_0}$ {\it with reduced structure}. Let 
$\Def_{\zeta_{A_0}}(\wt{X}_{A_0} ,\wt{H}_{A_0})\subset \Def(\wt{X}_{A_0} ,\wt{H}_{A_0})$ be the germ representing deformations  that  \lq\lq leave  $\zeta_{A_0}$ of type $(1,1)$\rq\rq. 
 The natural map of germs 
\begin{equation}\label{inclusione}
(\wt{\Sigma}_{A_0},A_0)\lra 
\Def_{\zeta_{A_0}}(\wt{X}_{A_0} ,\wt{H}_{A_0})
\end{equation}
is an inclusion because Map~\eqref{pescoluse} is an isomorphism.
Notice that $\zeta_{A_0}\in \wt{h}_{A_0}^{\bot}$  by~\eqref{ortogo}; since $\zeta_{A_0}\in H^{1,1}_{\ZZ}(\wt{X}_{A_0} )$ we have
\begin{equation}\label{inradice}
\wt{\cP}(\wt{\Sigma}_{A_0}) \subset 
\psi(\zeta_{A_0})^{\bot}\cap\Omega_{\Lambda}.
\end{equation}
Notice that $\zeta_{A_0}^{\bot}\cap\Omega_{\Lambda}$ has codimension $1$ and is smooth  because $(\zeta_{A_0},\zeta_{A_0})=-2$. By injectivity of the local period map we get injectivity of the period map restricted to $\wt{\Sigma}_{A_0}$:
\begin{equation}\label{sigmaiso}
\begin{matrix}
(\wt{\Sigma}_{A_0},A_0) & \hra & 
(\psi(\zeta_{A_0})^{\bot}\cap\Omega_{\Lambda},\psi_{\CC} H^{2,0}(\wt{X}_{A_0} ))\\
\\
t & \mapsto & \wt{\cP}(t)
\end{matrix}
\end{equation}
Since domain and codomain have equal dimensions
 the above map is a local isomorphism.   
In particular $d\wt{\cP}(A_0)$ is injective when restricted to the tangent space to $\wt{\Sigma}_{A_0}$ at $A_0$.  Thus it will suffice to exhibit a tangent vector $v\in T_{A_0}\wt{Z}_{A_0}$ such that $d\wt{\cP}(v)\notin\psi(\zeta_{A_0})^{\bot}$. By Item~(1) of~\Ref{crl}{batman} it suffices to prove that $E_{A_0}$ does not lift to $1$-st order in the direction $v$. Let $\Delta$ be the unit complex disc and $\gamma\colon\Delta\hra \wt{Z}_{A_0}$ be an inclusion with $v:=\gamma'(0)\notin\wt{\Sigma}_{A_0}$. Let $\wt{\cX}_{\Delta} \to\Delta$ be obtained by base-change from $g\colon\wt{\cX}_2 \to\cV$. Let $\PP^1$ be an arbitrary fiber of~(\ref{fibcon});  then $N_{\PP^1}\cX_{\Delta} \cong\cO_{\PP^1}(-1)\oplus \cO_{\PP^1}(-1)$.  It follows that $E_{A_0}$ does not lift to $1$-st order in the direction $v$. This finishes the proof that $d\wt{\cP}(\wt{A}_0)$ is injective. 
\end{proof}
\begin{prp}\label{prp:nerola}
The restriction of $\gp$ to $(\Sigma^{\rm sm}\setminus\Delta)// PGL(V)$ is a dominant map
 to  ${\mathbb S}_2^{\star}$ with finite fibers. Let $A\in(\Sigma^{\rm sm}\setminus\Delta)$ and $\psi$ be a marking of $(\wt{X} _A,\wt{H}_A)$: then $\psi(\zeta_A)$ is a $(-2)$-root of $\Lambda$ and $\divisore(\psi(\zeta_A))=1$. 
\end{prp}
\begin{proof}
Let $A\in(\Sigma^{\rm sm}\setminus\Delta)$. By~\Ref{clm}{vendola} we get that $[A]$ is an isolated point in the fiber $\gp^{-1}(\gp([A]))$.  In particular
\begin{equation}\label{mare}
\cod(\gp((\Sigma^{\rm sm}\setminus\Delta)// \PGL(V)),\DD_{\Lambda})=1.
\end{equation}
  By~\eqref{ortogo} and~\eqref{menodue}  $\psi(\zeta_A)$ is a $(-2)$-root of $\Lambda$. By~\eqref{inradice} and~\Ref{prp}{orbrad} we get that
\begin{equation}\label{alternativa}
\gp((\Sigma^{\rm sm}\setminus\Delta)// \PGL(V))\subset{\mathbb S}_2^{\star}\cup{\mathbb S}'_2\cup{\mathbb S}''_2.
\end{equation}
By~\eqref{mare} and irreducibility of $\Sigma$ the left-hand side of~\eqref{alternativa} is dense in  one of ${\mathbb S}_2^{\star}$, ${\mathbb S}'_2$, ${\mathbb S}''_2$. Let $\delta_V$ be as in~\eqref{specchio} and $\delta\colon\gM\to\gM$ be the induced involution, let $\ov{\iota}\colon\DD_{\Lambda}^{BB}\to\DD_{\Lambda}^{BB}$ be the involution given by~\eqref{invoper}). Then $(\Sigma// \PGL(V))$ is mapped  to itself  by  $\delta$   and hence~\eqref{idiota} gives that  its image under the period map $\gp$ is mapped to itself by $\ov{\iota}$. By~\eqref{azionesse} it follows that $\gp$ maps $(\Sigma^{\rm sm}\setminus\Delta)// PGL(V)$ 
 into  ${\mathbb S}_2^{\star}$ and hence that  $\divisore(\psi(\zeta_A))=1$. 
\end{proof}
\subsection{Periods of $K3$ surfaces of degree $2$}
\setcounter{equation}{0}
Let  $A\in(\Sigma^{\rm sm}\setminus\Delta)$. We will recall results of Shah on the period map for double covers   of $\PP^2$ branched over a sextic curve. 
Let
 $\gC_6:=|\cO_{\PP^2}(6)|// PGL_3$ 
 and $\Phi$ be the lattice given by~\eqref{retdue}. There is a period map 
\begin{equation}\label{kappadue}
\gs \colon \gC_6  \dashrightarrow  \DD_{\Phi}^{BB}
\end{equation}
whose restriction to the open set parametrizing smooth sextics is defined as follows. Let $C$ be a smooth plane sextic  and  $f\colon S\to\PP^2$ be the double cover branched over $C$.  Then~\eqref{kappadue} maps the orbit of $C$ to the period point of the polarized $K3$ surface $(S,f^{*}\cO_{\PP^2}(1))$.  Shah~\cite{shah} determined the \lq\lq boundary\rq\rq and the indeterminacy locus of the above map. In order to state Shah' results  we recall a definition.
\begin{dfn}\label{dfn:adesing}
A curve $C\subset\PP^2$ has a {\it simple singularity} at $p\in C$ if and only if the following hold:
\begin{itemize}
\item[(i)] 
$C$ is reduced in a neighborhood of $p$. 
\item[(ii)] 
$\mult _p(C)\le 3$ and if equality holds 
$C$ does not have a consecutive triple point at $p$.\footnote{$C$ has  a consecutive triple point at $p$ if the strict transform of $C$ in $Bl_p(\PP^2)$ has a point of multiplicity $3$ lying over $p$.} 
\end{itemize}
\end{dfn}
\begin{rmk}\label{rmk:semade}
Let $C\subset\PP^2$ be a sextic curve. Then $C$ has simple singularities if and only if   the double cover $S\to\PP^2$ branched over $C$ is a normal surface with DuVal singularities or equivalently  the minimal desingularization $\wt{S}$ of $S$ is a $K3$ surface (with A-D-E curves lying over the singularities of $S$), see Theorem 7.1 of~\cite{bpv}. 
\end{rmk}
Let $C\subset \PP^2$ be a sextic curve with simple singularities. Then $C$ is $PGL_3$-stable by~\cite{shah}.  We let
\begin{equation}\label{sestisem}
\gC^{ADE}_6:=\{C\in |\cO_{\PP^2}(6)| \mid \text{$C$ has simple singularities}\}// PGL_3.
\end{equation}
 Let $C$ be a plane sextic. If $C$  has simple singularities the period map~\eqref{kappadue}  is regular at $C$ and takes value in $\DD_{\Phi}$ - see~\Ref{rmk}{semade}. More generally Shah~\cite{shah} proved   that~\eqref{kappadue}  is regular at $C$  if and only if $C$ is $PGL_3$-semistable and the unique closed orbit in $\ov{PGL_3 C}\cap|\cO_{\PP^2}(6)|^{ss}$ is not that of triple (smooth) conics. 
\begin{dfn}
Let $\lagr^{ADE}\subset\lagr$  be the set of $A$ such that $C_{W,A}$ is a curve with simple singularities for every $W\in\Theta_A$.  
Let  $\lagr^{ILS}\subset\lagr$ be the set of $A$ such that the period map~\eqref{kappadue} is regular at $C_{W,A}$  for every $W\in\Theta_A$. 
\end{dfn}
Notice that both $\lagr^{ADE}$ and $\lagr^{ILS}$ are open. We have inclusions
\begin{equation}
(\lagr\setminus\Sigma)\subset\lagr^{ADE}\subset\lagr^{ILS}.
\end{equation}
The reason for the superscript $ILS$ is the following:   a curve $C\in|\cO_{\PP(W)}(6)|$ is in the regular locus of the period map~\eqref{persestiche} if and only if the double cover of $\PP(W)$ branched over $C$ has \emph{Insignificant Limit Singularities} in the terminology of Mumford, see~\cite{shah2}. 
\begin{dfn}
Let $\Sigma^{ILS}:=\Sigma\cap\lagr^{ILS}$. Let $\wt{\Sigma}^{ILS}\subset\wt{\Sigma}$ be the inverse image of $\Sigma^{ILS}$ for the natural forgetful map $\wt{\Sigma}\to\Sigma$, and  $\wh{\Sigma}^{ILS}\subset\wh{\Sigma}$ 
\begin{equation}
\wh{\Sigma}^{ILS}:=(p|_{\wh{\Sigma}})^{-1}(\Sigma^{ILS})
\end{equation}
 where $p\colon\wh{\lagr}\to \lagr$ and $\wh{\Sigma}$  are as in~\Ref{dfn}{cappello}. 
\end{dfn}
\subsection{The period map on $\Sigma$ and periods of $K3$ surfaces}
\setcounter{equation}{0}
Let  
\begin{equation}
\begin{matrix}
\wt{\Sigma}^{ILS} & \overset{\tau}{\lra} & \Sigma^{ILS} \\
(W,A) & \mapsto & A
\end{matrix}
\end{equation}
 be the forgetful map. Let $A\in(\Sigma^{\rm sm}\setminus\Delta)$: then $\Theta_A$ is a singleton by~\eqref{singsig} and if $W$ is the unique element of $\Theta_A$ then  $C_{W,A}$ is smooth sextic  by Item~(3) of~\Ref{crl}{runrabbit}. It follows that $(\Sigma^{\rm sm}\setminus\Delta)\subset\lagr^{ILS}$ and
$\tau$ defines an isomorphism $\tau^{-1}(\Sigma^{\rm sm}\setminus\Delta)\to (\Sigma^{\rm sm}\setminus\Delta)$. Thus we may regard $(\Sigma^{\rm sm}\setminus\Delta)$ as an (open dense) subset of    $\wt{\Sigma}^{ILS}$:
\begin{equation}\label{riccmonelli}
\iota\colon (\Sigma^{\rm sm}\setminus\Delta)\hra \wt{\Sigma}^{ILS}.
\end{equation}
By definition of $\wt{\Sigma}^{ILS}$ we have    a regular map
\begin{equation}
\begin{matrix}
\wt{\Sigma}^{ILS} & \overset{q}{\lra} & \DD_{\Phi}^{BB}\\
(W,A) & \mapsto & \Pi(S_{W,A},D_{W,A})
\end{matrix}
\end{equation}
where $D_{W,A}$ is the pull-back to $S_{W,A}$ of $\cO_{\PP(W)}(1)$ and $\Pi(S_{W,A},D_{W,A})$ is the (extended) period point of $(S_{W,A},D_{W,A})$. Recall that we have defined a finite map  $\rho\colon\DD^{BB}_{\Gamma}\to\DD^{BB}_{\Phi}$, see~\eqref{mapparo} and that there is a natural map $\nu\colon \DD^{BB}_{\Gamma}\to\ov{\mathbb S}_2^{\star}$ which is identified with the normalization of $\ov{\mathbb S}_2^{\star}$, see~\eqref{normesse}. 
\begin{prp}\label{prp:feniglia}
There exists a regular map 
\begin{equation}\label{mapperre}
Q\colon  \wt{\Sigma}^{ILS} \to  \DD_{\Gamma}^{BB}
\end{equation}
 such that $\rho\circ Q=q$. Moreover the composition $\nu\circ(Q|_{(\Sigma^{\rm sm}\setminus\Delta)})$ is equal to the restriction of the period map $\cP$ to  $(\Sigma^{\rm sm}\setminus\Delta)$.  
\end{prp}
\begin{proof}
By~\Ref{prp}{nerola} the restriction of the period map to  $(\Sigma^{\rm sm}\setminus\Delta)$  is a dominant map to $\ov{\mathbb S}_2^{\star}$   and therefore it lifts to the normalization of $\ov{\mathbb S}_2^{\star}$:
\begin{equation}\label{sollevo}
\xymatrix
{&  &   \DD^{BB}_{\Gamma}\ar^{\nu}[d]\\   
(\Sigma^{\rm sm}\setminus\Delta) \ar^{Q_0}[urr]  \ar^{\cP|_{(\Sigma^{\rm sm}\setminus\Delta)}}[rr] & & \ov{\mathbb S}_2^{\star} \\}
\end{equation}
We claim that
\begin{equation}\label{pescia}
\rho\circ Q_0=q|_{(\Sigma^{\rm sm}\setminus\Delta)}.
\end{equation}
In fact let $A\in (\Sigma^{\rm sm}\setminus\Delta)$. Let $r\colon \zeta_A^{\bot}\to H^2(S_A;\CC)$ be the isomorphism given by~\eqref{benhur}. 
 Let's prove that
\begin{equation}\label{leuca}
[H^2(S_A;\ZZ): r (\zeta_A^{\bot}\cap H^2(\wt{X}_A;\ZZ))]=2.
\end{equation}
In fact  $r$ is a homomorphism of lattices by~\Ref{clm}{conserva}. Since $H^2(S_A;\ZZ)$ and $\zeta_A^{\bot}\cap H^2(\wt{X}_A;\ZZ)$ have the same rank it follows that $r (\zeta_A^{\bot}\cap H^2(\wt{X}_A;\ZZ))$ is of finite index 
in  $H^2(S_A;\ZZ)$: let $d$ be the 
the index. By the last clause of~\Ref{prp}{nerola} the lattice  $(\zeta_A^{\bot} \cap H^2(\wt{X}_A;\ZZ))$ is isometric to $\wt{\Gamma}$ - see~\eqref{tilgam}. 
Hence we have 
\begin{equation}\label{vespa}
-4=\discr \wt{\Gamma}=\discr (\zeta_A^{\bot} \cap H^2(\wt{X}_A;\ZZ))=d^2\cdot \discr H^2(S_A;\ZZ)=-d^2.
\end{equation}
Equation~\eqref{leuca} follows at once. Next let $\psi\colon H^2(\wt{X} _A;\ZZ)\overset{\sim}{\lra}\wt{\Lambda}$ be a marking of $(\wt{X} _A,\wt{H}_A)$. By  the last clause of~\Ref{prp}{nerola} we know that $\psi(\zeta_A)$ is a $(-2)$-root of $\Lambda$ of divisibility $1$. By~\Ref{prp}{orbrad} there exists $g\in \wt{O}(\Lambda)$ such that $g\circ\psi(\zeta_A)=e_3$. Let $\phi:=g\circ\psi$. Then $\phi$ is a new marking of   $(\wt{X} _A,\wt{H}_A)$ and  $\phi(\zeta_A)=e_3$. It follows that $\phi(\zeta_A ^{\bot}\cap H^2(\wt{X}_A;\ZZ))=\wt{\Gamma}$. 
Let $\psi_{\QQ}\colon H^2(\wt{X} _A;\QQ)\overset{\sim}{\lra}\wt{\Lambda}_\QQ$ be the $\QQ$-linear extension of $\phi$.  By~\eqref{leuca} $H^2(S_A;\ZZ)$ is an overlattice of $\zeta_A^{\bot}\cap H^2(\wt{X}_A;\ZZ)$ and hence it may be emebedded canonically into $H^2(\wt{X} _A;\QQ)$: thus $\phi_\QQ(H^2(S_A;\ZZ))$ makes sense. By~\eqref{leuca} we get that  $\phi_\QQ(H^2(S_A;\ZZ))$ is an  overlattice of $\phi (\zeta_A^{\bot} \cap H^2(\wt{X}_A;\ZZ))$ and that $\phi (\zeta_A^{\bot}  \cap H^2(\wt{X}_A;\ZZ))$ has index $2$ in $\phi_\QQ(H^2(S_A;\ZZ))$. By~\Ref{clm}{unisov} it follows that $\phi_\QQ(H^2(S_A;\ZZ))=\wt{\Phi}$. Equation~\eqref{pescia} follows at once from this.
By~\eqref{pescia} we have a  commutative diagram 
\begin{equation}\label{trapezio}
\xymatrix{  & \wt{\Sigma}^{ILS} \times_{\DD_{K_2}^{BB}}\DD_{\Gamma}^{BB} \ar[d]  \ar[r]  & \DD_{\Gamma}^{BB}
 \ar^{\rho}[d] \\   
 (\Sigma^{\rm sm}\setminus\Delta) \ar^{(\iota,Q_0)}[ur] \ar^{\iota}[r] & \wt{\Sigma}^{ILS}  
 \ar^{q}[r] &  \DD_{\Phi}^{BB}.}
\end{equation}
where $\iota$ is the inclusion map~\eqref{riccmonelli}. 
Let $\cZ$ be the closure of $\im(\iota,Q_0)$. Then  $\cZ$ is an irreducible component of $\wt{\Sigma}^{ILS}\times_{\DD_{K_2}^{BB}}\DD_{\Gamma}^{BB}$ because $\iota$ is an open inclusion. The natural projection 
$\cZ\to \wt{\Sigma}^{ILS}$
is a finite birational map and hence it is an isomorphism because $\wt{\Sigma}^{ILS}$ is smooth. We define the map $Q\colon \wt{\Sigma}^{ILS} \to  \DD_{\Gamma}^{BB}$ as the composition of the inverse $\wt{\Sigma}^{ILS}\to \cZ$ and the projection $\cZ\to \DD_{\Gamma}^{BB}$. The properties of $Q$ stated in the proposition hold  by commutativity of~\eqref{trapezio}.  
\end{proof}
\begin{crl}\label{crl:feniglia}
The image of the map $(\tau,\nu\circ Q)\colon\wt{\Sigma}^{ILS}\to \Sigma^{ILS}\times\ov{\mathbb S}_2^{\star}$ is equal 
to $\wh{\Sigma}^{ILS}$. 
\end{crl}
\begin{proof}
Let  $p\colon\wh{\lagr}\to \lagr$ be as in~\Ref{dfn}{cappello}. Since $\cP$ is regular on $(\Sigma^{\rm sm}\setminus \Delta)$ the map $p^{-1}(\Sigma^{\rm sm}\setminus \Delta)\to 
(\Sigma^{\rm sm}\setminus \Delta)$ is an isomorphism and $p^{-1}(\Sigma^{\rm sm}\setminus \Delta)$ is an open dense subset of $\wh{\Sigma}^{ILS}$ (recall that $\Sigma$ is irreducible and hence so is $\wh{\Sigma}$). 
By the second clause of~\Ref{prp}{feniglia} we have that 
\begin{equation}\label{tombolo}
(\tau,\nu\circ Q)(\Sigma^{\rm sm}\setminus \Delta)= p^{-1}(\Sigma^{\rm sm}\setminus \Delta).
\end{equation}
Since  $\wh{\Sigma}$ is closed in $\lagrhat\times\DD_{\Lambda}^{BB}$ it follows that $\im(\tau,\nu\circ Q)\subset \wh{\Sigma}$. The commutative diagram
\begin{equation}\label{frigidaire}
\xymatrix
{\wt{\Sigma}^{ILS}  \ar^{(\tau,\nu\circ Q)}[rr] \ar_{\tau}[dr] &  &  \wh{\Sigma}^{ILS} \ar^{p|_{\wh{\Sigma}^{ILS}}}[dl] \\   
& {\Sigma}^{ILS}  &  }
\end{equation}
gives that $\im(\tau,\nu\circ Q)\subset \wh{\Sigma}^{ILS}$. The right-hand side of~\eqref{tombolo} is dense in $\wh{\Sigma}^{ILS}$: thus in order to finish the proof it suffices to show that $\im(\tau,\nu\circ Q)$ is closed in $\wh{\Sigma}^{ILS}$. The equality $(p|_{\wh{\Sigma}^{ILS}})\circ(\tau,\nu\circ Q)=\tau$ and  properness of $\tau$ give that $(\tau,\nu\circ Q)$ is proper (see Ch.~II, Cor.~4.8, Item~(e) of~\cite{hart}) and hence closed: thus $\im(\tau,\nu\circ Q)$ is closed in $\wh{\Sigma}^{ILS}$. 
\end{proof}
\subsection{Extension of the period map}\label{subsec:ariestendi}
\setcounter{equation}{0}
\begin{prp}\label{prp:regper}
Let  $A\in\lagr^{ILS}$. If $\dim\Theta_A\le 1$ the  period map $\cP$ is regular at $A$ and moreover  $\cP(A)\in\DD_\Lambda$ if and only if $A\in \lagr^{ADE}$.
\end{prp}
\begin{proof}
If $A\notin\Sigma$ then $\cP$ is regular at $A$ by~\Ref{prp}{dagest}. Now assume that $A\in\Sigma^{ILS}$. 
Suppose that $\cP$  is not regular at $A$: we will reach a contradiction.  Let $p\colon\lagrhat\to \lagr$ be as in~\Ref{dfn}{cappello}. Then $p^{-1}(A)\cap\wh{\Sigma}$ is a subset of $\{A\}\times\DD^{BB}_{\Lambda}$ and hence we may  identify it with its projection in   $\DD^{BB}_{\Lambda}$. This subset is  equal to $\nu\circ Q(\tau^{-1}(A))$ by~\Ref{crl}{feniglia} and  Commutative Diagram~\eqref{frigidaire}. 
On the other hand  $\tau^{-1}(A)=\Theta_A$ and hence $\dim\tau^{-1}(A)\le 1$ by hypothesis:  it follows that $\dim p^{-1}(A)\le 1$ and this contradicts~\Ref{crl}{dagest}. 
 This proves that $\cP$ is regular at $A$. The last clause of the proposition follows  from~\Ref{crl}{feniglia}.
\end{proof}
\noindent
{\it Proof of~\Ref{thm}{teorprinc}.\/}
Let  $x\in(\gM\setminus\gI)$. There exists a GIT-semistable $A\in\lagr$ representing $x$ with $\PGL(V)$-orbit closed in the semistable locus $\lagr^{\rm ss}$, and such $A$ is determined up to the action of $\PGL(V)$. By Luna's  \'etale slice Theorem~\cite{luna1} it suffices to prove that the period map $\cP$ is regular at $A$. 
If $A\notin\Sigma$ then $\cP$ is regular at $A$ and $\cP(A)\in\DD_{\Lambda}$ by~\Ref{prp}{dagest}. Now suppose that $A\in\Sigma$.  Then  $A\in\Sigma^{ILS}$  because  $x\notin\gI$.  
By~\Ref{prp}{regper} in order to prove that $\cP$ is regular at $A$ it will suffice to show that $\dim\Theta_A\le 1$. Suppose that $\dim\Theta_A\ge 2$, we will reach a contradiction. Theorem~3.26 and Theorem~3.36 of~\cite{ogtasso} give that $A$ belongs to certain subsets of $\lagr$, namely $\XX_{\cA_{+}},\XX_{\cA^{\vee}_{+}},\ldots,\XX_{+}$ (notice the misprint in the statement of Theorem~3.36: $\XX_{\cD}$ is to be replaced by $\XX_{\cD,+}$). Thus, unless 
\begin{equation}\label{licaone}
A\in(\XX_{\cY}\cup \XX_{\cW}\cup \XX_{h}\cup \XX_{k}),
\end{equation}
(notice that $\XX_{+} \subset \XX_{\cW}$)  the lagrangian $A$ belongs to one of the standard unstable strata listed in Table 2 of~\cite{ogmoduli}, and hence is $\PGL(V)$-unstable. That is a contradiction because 
  $A$ is semistable and hence we conclude that~\eqref{licaone} holds. Proposition~4.3.7 of~\cite{ogmoduli} gives that if $A\in\XX_{\cY}$ then $A\in\PGL_6 A_{+}$ i.e.~$A\in\XX_{+}$ (because $A$ has minimal $\PGL(V)$-orbit), thus $A\in \XX_{\cW}\cup \XX_{h}\cup \XX_{k})$. If  $A\in \XX_{\cW}$ then $A\notin\Sigma^{ILS}$ by Claim~4.4.5 of~\cite{ogmoduli}, if $A\in \XX_{h}$ then  $A\notin\Sigma^{ILS}$ by~(4.5.6) of~\cite{ogmoduli}, if $A\in \XX_{k}$ then  $A\notin\Sigma^{ILS}$ by~(4.5.5) of~\cite{ogmoduli}:  that is a contradiction. This shows that  $\dim\Theta_A\le 1$ and hence $\gp$ is regular at $x$.
The last clause of~\Ref{prp}{regper} gives that $\gp(x)\in\DD_{\Lambda}$ if and only if $x\in\gM^{ADE}$.
\qed
 \section{On the image of the period map}\label{sec:imagoper}
 \setcounter{equation}{0}
We will prove~\Ref{thm}{imagoper}. Most of the work goes into showing that $\cP(\lagr\setminus\Sigma)$ does not intersect ${\mathbb S}_2^{\star}\cup{\mathbb S}'_2\cup{\mathbb S}''_2\cup{\mathbb S}_4$ i.e.~that $\gp$ maps $(\gM\setminus\gN)$ into the right-hand side of~\eqref{looicomp}. First we will prove that result, then we will show that the restriction of $\gp$ to $(\gM\setminus\gN)$ is an open embedding.
 \subsection{Proof that $\cP(\lagr\setminus\Sigma)\cap{\mathbb S}_2^{\star}=\es$}
 \setcounter{equation}{0}
Let $S$ be a $K3$ surface. We recall~\cite{beau} the description of $H^2(S^{[2]})$ and the Beauville-Bogomolov form $q_{S^{[2]}}$ in terms of $H^2(S)$. Let $\mu\colon H^2(S)\to H^2(S^{[2]})$ be the composition of the  symmetrization map $H^2(S)\to H^2(S^{(2)})$ and the pull-back 
$H^2(S^{(2)})\to H^2(S^{[2]})$. There is a direct sum decomposition 
\begin{equation}\label{argentieri}
H^2(S^{[2]})=\mu(H^2(S;\ZZ))\oplus \ZZ\xi
\end{equation}
 where $2\xi$ is represented by the locus parametrizing non-reduced subschemes. Moreover if $H^2(S)$ and $H^2(S^{[2]})$ are equipped with the intersection form and Beauville-Bogomolov quadratic form $q_{S^{[2]}}$ respectively, then  $\mu$ is an isometric embedding, Decomposition~\eqref{argentieri} is orthogonal, and $q_{S^{[2]}}(\xi)=-2$. Recall that $\delta_V\colon\lagr\overset{\sim}{\lra}\lagrdual$ is defined by $\delta_V(A):=\Ann A$, see~\eqref{specchio}.
\begin{lmm}\label{lmm:nondel}
 $\cP(\Delta\setminus\Sigma)\not\subset({\mathbb S}_2^{\star}\cup{\mathbb S}'_2\cup{\mathbb S}''_2\cup{\mathbb S}_4)$ and 
 $\cP(\delta_V(\Delta)\setminus\Sigma)\not\subset({\mathbb S}_2^{\star}\cup{\mathbb S}'_2\cup{\mathbb S}''_2\cup{\mathbb S}_4)$. 
\end{lmm}
\begin{proof}
Let $A\in(\Delta\setminus\Sigma)$ be  generic. By Theorem~4.15 of~\cite{ogdoppio} there exist a projective $K3$ surface $S_A$ of genus $6$  and a small contraction $S_A^{[2]}\to X_A$. Moreover the  period point $\cP(A)$  may be identified with the Hodge structure of $S_A^{[2]}$ as follows. 
The surface $S_A$ comes equipped with an ample divisor $D_A$ of genus $6$ i.e.~$D_A\cdot D_A=10$, let $d_A$ be the Poincar\`e dual of $D_A$. 
Then $\cP(A)$ is identified with the Hodge structure on $(\mu(D_A)-2\xi)^{\bot}$, where $\xi$ is as above. By Proposition~4.14  of~\cite{ogdoppio} we may assume that $(S_A,D_A)$ is a general polarized $K3$ surface of genus $6$. It follows that if $A$ is very general in $(\Delta\setminus\Sigma)$ then  $H^{1,1}_{\ZZ}(S_A)=\ZZ d_A$. Thus for  $A\in(\Delta\setminus\Sigma)$  very general we have that
\begin{equation}\label{generatore}
H^{1,1}_{\ZZ}(S_A^{[2]})\cap (\mu(d_A)-2\xi)^{\bot}=\ZZ(2\mu(d_A)-5\xi).
\end{equation}
 Now suppose that $\cP(A)\in({\mathbb S}_2^{\star}\cup{\mathbb S}'_2\cup{\mathbb S}''_2\cup{\mathbb S}_4)$: by definition  there exists $\alpha\in H^{1,1}_{\ZZ}(S_A^{[2]})\cap (\mu(d_A)-2\xi)^{\bot}$ of square $(-2)$ or $(-4)$: since $q_{S^{[2]}_A}(2\mu(d_A)-5\xi)=-10$ that contradicts~\eqref{generatore}. This proves that $\cP(\Delta\setminus\Sigma)\not\subset({\mathbb S}_2^{\star}\cup{\mathbb S}'_2\cup{\mathbb S}''_2\cup{\mathbb S}_4)$.  Next let $\ov{\iota}\colon \DD_{\Lambda}^{BB}\to \DD_{\Lambda}^{BB}$ be the involution defined by $\delta_V$, see~\eqref{invoper}.  Then $\ov{\iota}$ maps  $({\mathbb S}_2^{\star}\cup{\mathbb S}'_2\cup{\mathbb S}''_2\cup{\mathbb S}_4)$ to itself, see~\eqref{azionesse}, and hence $\cP(\delta_V(\Delta)\setminus\Sigma)\not\subset({\mathbb S}_2^{\star}\cup{\mathbb S}'_2\cup{\mathbb S}''_2\cup{\mathbb S}_4)$ because otherwise it would follow that $\cP(\Delta\setminus\Sigma)\subset({\mathbb S}_2^{\star}\cup{\mathbb S}'_2\cup{\mathbb S}''_2\cup{\mathbb S}_4)$. 
\end{proof}
Suppose that $\cP(\lagr\setminus\Sigma)\cap{\mathbb S}_2^{\star}\not=\es$. Since ${\mathbb S}_2^{\star}$ is a $\QQ$-Cartier divisor of $\DD_{\Lambda}$ it follows that $\cP^{-1}({\mathbb S}_2^{\star})\cap (\lagr\setminus\Sigma)$ has pure codimension $1$:  let $C$ be one of its 
 irreducible components. Then  $C\not=\Delta$ by~\Ref{lmm}{nondel} and hence $C^0:=C\setminus\Delta$ is a codimension-$1$ $\PGL(V)$-invariant closed subset of $(\lagr\setminus\Sigma)$. 
  Since $C^0$ has pure codimension $1$ and is contained in the stable locus of $\lagr$ (see Cor.~2.5.1 of~\cite{ogmoduli}) the quotient $C^0//\PGL(V)$ has codimension $1$ in $\gM$. If $A\in (C^0\setminus\Delta)$ then $X_A$ is smooth and hence the local period map $\Def(X_A,H_A)\to\Omega_{\Lambda}$ is a (local) isomorphism (local Torelli for hyperk\"ahler manifolds): it follows that  the restriction of $\gp$ to $C^0//\PGL(V)$ has  finite  fibers and hence  $\cP(C^0)$ is dense in ${\mathbb S}_2^{\star}$. Now consider the period map $\gp\colon (\gM\setminus\gI)\to\DD^{BB}_{\Lambda}$: it is birational by Verbitsky's Global Torelli and Markman's monodromy results, see Theorem~1.3 and Lemma~9.2 of~\cite{markman}. We have proved that there are (at least) two distinct components in $\gp^{-1}(\ov{\mathbb S}_2^{\star})$ which are mapped \emph{dominantly} to $\ov{\mathbb S}_2^{\star}$ by $\gp$, namely $((\Sigma//\PGL(V))\setminus\gI)$ and the closure of $C^0//\PGL(V)$: that is a contradiction because $\DD^{BB}_{\Lambda}$ is normal. 
\qed
 \subsection{Proof that $\cP(\lagr\setminus\Sigma)\cap({\mathbb S}'_2\cup{\mathbb S}''_2)=\es$}
 \setcounter{equation}{0}
First we will prove that $\cP(\lagr\setminus\Sigma)\cap{\mathbb S}''_2=\es$. Let $U$ be a $3$-dimensional complex vector space and $\pi\colon S\to\PP(U)$ a double cover branched over a smooth sextic curve; thus $S$ is a $K3$ surface. Let $D\in |\pi^{*}\cO_S(1)|$ and $d\in H^{1,1}_{\ZZ}(S;\ZZ)$ be its Poincar\'e dual. Since $S^{[2]}$ is simply connected there is a unique class $\mu(D)\in\Pic(S^{[2]})$ whose first Chern class is equal to $\mu(d)$. One easily checks the following facts. There is a natural isomorphism
\begin{equation}
 \PP(\Sym^2 U^{\vee}) \overset{\sim}{\lra} |\mu(D)|
\end{equation}
and  the composition of the natural maps
\begin{equation}\label{ippogrifo}
S^{[2]}\lra S^{(2)}\lra \PP(U)^{(2)}\lra \PP(\Sym^2 U)
\end{equation}
is identified with the natural map $f\colon S^{[2]}\to |\mu(D)|^{\vee}$. 
The image of $f$ is  the chordal variety $\cV_2$ of  the Veronese surface $\{[u^2] \mid 0\not=u\in U\}$, and the map $ S^{[2]}\to\cV_2$ is finite of degree $4$. Since $\mu(d)$ has square $2$ we have a well-defined period point $\Pi(S^{[2]},\mu(d))\in \DD_{\Lambda}$. The class
  $\xi\in H^{1,1}_{\ZZ}(S^{[2]})$ is a $(-2)$-root of divisibility $2$ and it is orthogonal to $\mu(d)$: it follows that  $\Pi(S^{[2]},\mu(d))\in({\mathbb S}'_2\cup{\mathbb S}''_2)$. Actually $\Pi(S^{[2]},\mu(d))\in{\mathbb S}''_2$ because the divisibility of $\xi$ as an element  of $H^2(S^{[2]};\ZZ)$   is equal to $2$ (and not only as element of $\mu(d)^{\bot}$). The periods $\Pi(S^{[2]},\mu(d))$ with $S$  as above fill-out an open dense subset of ${\mathbb S}''_2$. Now suppose that there exists $A\in(\lagr\setminus\Sigma)$ such that $\cP(A)\in{\mathbb S}''_2$.  
Since ${\mathbb S}''_2$ is a $\QQ$-Cartier divisor of $\DD_{\Lambda}$ it follows that $\cP^{-1}({\mathbb S}''_2)\cap (\lagr\setminus\Sigma)$ has pure codimension $1$:  let $C$ be one of its 
 irreducible components. By~\Ref{lmm}{nondel} $C^0:=(C\setminus\Delta)$ is a codimension-$1$ subset of $\lagr^0$ and hence $\cP(C^0)$ contains an open dense subset of  ${\mathbb S}''_2$. It follows that there exist $A\in C^0$ and  a double cover $\pi\colon S\to\PP(U)$ as above \emph{with $\Pic(S)=\ZZ\mu(D)$} and such that $\cP(A)=\Pi(S^{[2]},\mu(d))$. By Verbitsky's Global Torelli Theorem there exists a birational map $\varphi\colon S^{[2]}\dra X_A$. Now $\varphi^{*}h_A$ is a $(1,1)$-class of square $2$: since  $\Pic(S)=\ZZ\mu(D)$ it follows that $\varphi^{*}h_A=\pm\mu(d)$, and hence $\varphi^{*}H_A=\mu(D)$ because $|-\mu(D)|$ is empty. But that is a contradiction because the map $f_A\colon X_A\to |H_A|^{\vee}$ is $2$-to-$1$ onto its image while the map $f\colon S^{[2]}\to |\mu(D)|^{\vee}$ has degree $4$ onto its image. This proves that 
\begin{equation}\label{primopasso}
\cP(\lagr\setminus\Sigma)\cap{\mathbb S}''_2=\es.
\end{equation}
 It remains to prove that $\cP(\lagr\setminus\Sigma)\cap{\mathbb S}'_2=\es$. Suppose that $\cP(\lagr\setminus\Sigma)\cap{\mathbb S}'_2\not=\es$. Let $\Sigma(V^{\vee})$ be the locus of $A\in\lagrdual$ containing a non-zero decomposable tri-vector.  Since $\delta_{V}(\Sigma)=\Sigma(V^{\vee})$ we get that  
 $\cP(\lagrdual\setminus\Sigma(V^{\vee}))\cap{\mathbb S}''_2\not=\es$ by~\eqref{idiota} and~\eqref{azionesse}: that contradicts~\eqref{primopasso}.
 \qed
\begin{rmk}
In the above proof we have noticed that the generic point of ${\mathbb S}''_2$ is equal to $\Pi(S^{[2]},\mu(d))$. One may also identify explicitly polarized hyperk\"ahler varieties whose periods belong to   ${\mathbb S}'_2$. In fact let $\pi\colon S\to \PP(U)$  and 
$D$,   $d$ be as above. Let $v\in H^{*}(S;\ZZ)$ be the Mukai vector $v:=(0,d,0)$ and let $\cM_v$ be the corresponding moduli space of $D$-semistable sheaves on $S$ with Mukai vector $v$: the generic such sheaf  is isomorphic to $\iota_{*}\eta$ where $\iota\colon C\hra S$ is the inclusion of a smooth $C\in|D|$ and $\eta$ is an invertible sheaf on $C$ of degree $1$. As is well-known $\cM_v$ is a hyperk\"ahler variety deformation equivalent to $K3^{[2]}$. Moreover $H^2(\cM_v)$ with its Hodge structure and B-B form  is identified with $v^{\bot}$ with the Hodge structure it inherits from the Hodge structure of $H^{*}(S)$ and the quadratic form given by the Mukai pairing, see~\cite{yoshi}.  Let  $h\in H^2(\cM_v)$  correspond to $\pm(1,0,-1)$. Then $h$ has square $2$ and, as is easily checked, the period point of $(\cM_v,h)$ belongs to   ${\mathbb S}'_2$: more precisely $\Pi(\cM_v,h)=\ov{\iota}(\Pi(S^{[2]},\mu(d)))$.  
\end{rmk}
 \subsection{Proof that $\cP(\lagr\setminus\Sigma)\cap{\mathbb S}_4=\es$}
 \setcounter{equation}{0}
Let $S\subset\PP^3$ be a smooth quartic surface, $D\in|\cO_S(1)|$ and $d$ be the Poincar\`e dual of $D$. We have a natural map
\begin{equation}
\begin{matrix}
S^{[2]} & \overset{f}{\lra} & \Gr(1,\PP^3)\subset\PP^5 \\
Z & \mapsto & \la Z\ra
\end{matrix}
\end{equation}
where $ \la Z\ra$ is the unique line containing the lenght-$2$ scheme $Z$. Let $H\in |f^{*}\cO_{\PP^5}(1)|$ and $h$ be its Poincar\`e dual. One checks easily that $h=(\mu(d)-\xi)$, in particular $q_{S^{[2]}}(h)=2$. Moreover pull-back gives an identification of $f$ with the natural map $S^{[2]}\to | H|^{\vee}$. The equalities
\begin{equation}
(h,\mu(d)-2\xi)_{S^{[2]}}=0,\qquad q_{S^{[2]}}(\mu(d)-2\xi)=-4,\qquad (h^{\bot},\mu(d)-2\xi)_{S^{[2]}}=2\ZZ
\end{equation}
(here $h^{\bot}\subset H^2(S^{[2]};\ZZ)$ is the subgroup of classes orthogonal to $h$) show that $(\mu(d)-2\xi)$ is a $(-4)$-root of $h^{\bot}$ and hence $\Pi(S^{[2]},h)\in{\mathbb S}_4$ by~\Ref{prp}{orbrad}. Moreover the generic point of ${\mathbb S}_4$ is equal to $\Pi(S^{[2]},h)$ for some $(S,d)$ as above: in fact ${\mathbb S}_4$ is irreducible, see~\Ref{rmk}{essirr}, of dimension $19$ i.e.~the dimension  of the set of periods $\Pi(S^{[2]},h)$ for $(S,d)$ as above. Now assume that $\cP(\lagr\setminus\Sigma)\cap{\mathbb S}_4\not=\es$. Arguing as in the previous cases we get that there exists a closed $\PGL(V)$-invariant codimension-$1$ subvarietry $C^0\subset(\lagr\setminus\Delta\setminus\Sigma)$ such that $\cP(C^0)\subset {\mathbb S}_4$. Thus $\cP(C^0)$ contains an open dense subset of ${\mathbb S}_4$ and therefore if $A\in C^0$ is very generic $h^{1,1}_{\ZZ}(X_A)=2$. By the discussion above we get that there exist $(S,d)$ as above such that $\Pi(S^{[2]},h)=\Pi(X_A,h_A)$ with $h^{1,1}_{\ZZ}(X_A)=2$. By Verbitsky's Global Torelli Theorem there exists a birational map
$\varphi\colon S^{[2]}\dra X_A$. Since the map $f_A\colon X_A\to |H_A|^{\vee}$ is of degree $2$ onto its image, and since $\varphi$ defines an isomorphism between the complement of a codimension-$2$ subsets of $S^{[2]}$ and  the complement of a codimension-$2$ subsets of $X_A$ (because both are varieties with trivival canonical bundle) we get that 
\begin{equation}\label{esticazzi}
\text{$q_{S^{[2]}}(\varphi^{*}h_A)=2$,  $ |\varphi^{*}H_A|$ has no base divisor, $S^{[2]}\dra |\varphi^{*}H_A|^{\vee}$ is of degree $2$ onto its image.}
\end{equation}
We will get a contradiction by showing that there exists no divisor of square $2$ on $S^{[2]}$ such that~\eqref{esticazzi} holds. Notice that if $H$ is the divisor  on $S^{[2]}$   defined above then the first two conditions of~\eqref{esticazzi} hold but not the third (the degree of the map is equal to $6$). This does not finish the proof because the set of elements of $ H^{1,1}_{\ZZ}(S^{[2]})$ whose square is $2$ is infinite.
\begin{lmm}\label{lmm:quadue}
There exists $n\in\ZZ$ such that
\begin{equation}\label{infostud}
\varphi^{*}h_A=x\mu(d)+y\xi,\qquad y+x\sqrt{2}=(-1+\sqrt{2})(3+2\sqrt{2})^n.
\end{equation}
\end{lmm}
\begin{proof}
Since $h^{1,1}_{\ZZ}(X_A)=2$ we have $h^{1,1}_{\ZZ}(S^{[2]})=2$ and hence $H^{1,1}_{\ZZ}(S^{[2]})$ is generated (over $\ZZ$) by $\mu(d)$ and $\xi$.  Let
\begin{equation}\label{triciclo}
\begin{matrix}
H^{1,1}_{\ZZ}(S^{[2]}) & \overset{\psi}{\lra} & \ZZ[\sqrt{2}] \\
x\mu(d)+y\xi & \mapsto & y+x\sqrt{2}
\end{matrix}
\end{equation}
Then 
\begin{equation}\label{prodtraccia}
(\alpha,\beta)=-\Tr(\psi(\alpha)\cdot\ov{\psi(\beta)}).
\end{equation}
 Since $\varphi^{*}h_A$ is an element of square $2$ we will need to solve a (negative) Pell equation. Solving Pell's equation $N(y+x\sqrt{2})=1$ (see for example Proposition~17.5.2 of~\cite{rosen}) and noting that $N(-1+\sqrt{2})=-1$ 
we get that there exists $n\in\ZZ$ such that
\begin{equation}
\varphi^{*}h_A=x\mu(d)+y\xi,\qquad y+x\sqrt{2}=\pm(-1+\sqrt{2})(3+2\sqrt{2})^n.
\end{equation}
Next notice   that $S$ does not contains lines  because   $h^{1,1}_{\ZZ}(S)=1$: it follows that the map $S^{[2]}\to \Gr(1,\PP^3)$ is finite and therefore $H$ is ample. 
Since $|\varphi^{*}H_A|$ is not empty and  $\varphi^{*}H_A$ is not equivalent to $0$ we get that 
\begin{equation}
0<(\varphi^{*}h_A,h)_{S^{[2]}}=-\Tr\left(\pm(-1+\sqrt{2})(3+2\sqrt{2})^n(-1-\sqrt{2})\right).
\end{equation}
It follows that the $\pm$ is actually $+$.
\end{proof}
Next we will consider the analogue of nodal classes on $K3$ surfaces. For $n\in\ZZ$ we define  $\alpha_n\in H^{1,1}_{\ZZ}(S^{[2]})$ by requiring that 
\begin{equation}
\psi(\alpha_n)=-(3-2\sqrt{2})^n.
\end{equation}
Thus $q_{S^{[2]}}(\alpha_n)=-2$ for all $n$. 
\begin{lmm}\label{lmm:nodeff}
If $n> 0$ then $2\alpha_n$ is effective, if  $n\le 0$ then $-2\alpha_n$ is effective.
\end{lmm}
\begin{proof}
By Theorem~1.11 of~\cite{markman2} either $2\alpha_n$ or $-2\alpha_n$ is effective (because $q_{S^{[2]}}(\alpha_n)=-2$). Since $(\mu(d)-\xi)$ is ample we decide which of $\pm 2\alpha_n$  is effective by requiring that the product with $(\mu(d)-\xi)$ is strictly positive. The result  follows easily from~\eqref{prodtraccia}.
\end{proof}
\begin{prp}\label{prp:contiene}
Suppose that $\varphi^{*}h_A$ is given by~\eqref{infostud} with $n\not=0$. Then there exists an \emph{effective} $\beta\in H^{1,1}_{\ZZ}(S^{[2]})$ such that  $(\varphi^{*}h_A,\beta)_{S^{[2]}}<0$. 
\end{prp}
\begin{proof}
Identify $H^{1,1}_{\ZZ}(S^{[2]}) $ with $\ZZ[\sqrt{2}]$ via~\eqref{triciclo} and let $g\colon H^{1,1}_{\ZZ}(S^{[2]}) \to H^{1,1}_{\ZZ}(S^{[2]})$ correspond to multiplication by $(3-2\sqrt{2})$. Since $N(3-2\sqrt{2})=1$ the map $g$ is an isometry. Notice that $\alpha_k=g^k(-\xi)$ and    by~\Ref{lmm}{quadue} we have that $\varphi^{*}h_A=g^{-n}(\mu(d)-\xi)$. Now suppose that $n>0$. Then $-2\alpha_{-n+1}$ is effective by~\Ref{lmm}{nodeff} and 
\begin{equation}\label{asilo}
\scriptstyle
(\varphi^{*}h_A,-2\alpha_{-n+1})_{S^{[2]}}=(g^{-n}(\mu(d)-\xi),2g^{-n+1}(\xi))_{S^{[2]}}=(\mu(d)-\xi,2g(\xi))_{S^{[2]}}
=(\mu(d)-\xi,-4\mu(d)+6\xi)_{S^{[2]}}=-4<0.
\end{equation}
Lastly suppose that $n<0$. Then $2\alpha_{-n}$ is effective by~\Ref{lmm}{nodeff} and 
\begin{equation}\label{asilo}
(\varphi^{*}h_A,2\alpha_{-n})_{S^{[2]}}=(g^{-n}(\mu(d)-\xi),2g^{-n}(-\xi))_{S^{[2]}}=(\mu(d)-\xi,-2\xi)_{S^{[2]}}
=-4<0.
\end{equation}
\end{proof}
Now we are ready to prove that~\eqref{esticazzi} cannot hold and hence reach a contradiction. By~\Ref{lmm}{quadue} we know that $\varphi^{*}h_A$ is given  by~\eqref{infostud} for some  $n\in\ZZ$. We have already noticed that ~\eqref{esticazzi} cannot hold if $n=0$ . Suppose that $n\not=0$. By~\Ref{prp}{contiene} there exists an effective $\beta\in H^{1,1}_{\ZZ}(S^{[2]})$ such that 
\begin{equation}\label{negativo}
(\varphi^{*}h_A,\beta)_{S^{[2]}}<0. 
\end{equation}
Let $B$ be an effective divisor representing $\beta$ and $C\in | \varphi^{*}h_A |$. Then $C\cap B$ does not have codimension $2$ i.e.~there exists at least one prime divisor $B_i$ which is both in the support of $B$ and in the support of $C$. In fact suppose the contrary. Let $c\in H^2(S^{[2]})$ be the Poincar\`e dual of $C$ and $\sigma$ be a symplectic form on $S^{[2]}$: then
\begin{equation}
0\le \int_{B\cap C}\sigma\wedge\ov{\sigma}=(\beta,c)_{S^{[2]}}(\sigma,\ov{\sigma})_{S^{[2]}}
\end{equation}
and since $(\sigma,\ov{\sigma})_{S^{[2]}}>0$ we get that $(\beta,c)\ge 0$ i.e.~$(\varphi^{*}h_A,\beta)_{S^{[2]}}\ge 0$, that  contradicts~\eqref{negativo}. The conclusion is that there exists a prime divisor $B_i$ which is both in the support of $B$ and of \emph{any} $C\in | \varphi^{*}h_A |$, i.e.~$B_i$ is  a base divisor of the linear system $| \varphi^{*}h_A |$: that shows that~\eqref{esticazzi}  does not hold.
\qed
 \subsection{Proof that $\gp$ restricted to $(\gM\setminus\gN)$ is an open embedding.}\label{subsec:immaperta}
 \setcounter{equation}{0}
Let $\gT:=\Delta//\PGL(V)$. Let $\delta\colon\gM\to\gM$ be the duality involution defined 
by~\eqref{specchio}. By~\cite{ogduca}, pp.~36-38, we have that
\begin{equation}\label{codal}
\cod(\gT\cap\delta(\gT),\gM)\ge 2.
\end{equation}
\begin{clm}\label{clm:granaper}
The restriction of $\gp$ to $(\gM\setminus\gN\setminus(\gT\cap\delta(\gT)))$ is open.
\end{clm}
\begin{proof}
It suffices to prove that the restriction of $\gp$ to $(\gM\setminus\gN\setminus(\gT\cap\delta(\gT)))$ is open in the classical topology. Suppose first that  $[A]\in(\gM\setminus\gN\setminus\gT)$ i.e.~$X_A$ is smooth. Then $A$ is stable by Corollary~2.5.1 of~\cite{ogmoduli}. Let $Z_A\subset\lagr$ be an analytic $\PGL(V)$-slice at $A$, see~\cite{luna1}. We may and will assume that $Z_A$ is contractible; hence a marking of $(X_A,H_A)$ defines a lift of   $\cP|_{Z_A}$ to a regular map $\wt{\cP}_A\colon Z_A\to\Omega_{\Lambda}$.  The family of double EPW-sextics parametrized by $Z_A$ is a representative of the universal deformation space of the polarized $4$-fold $(X_A,H_A)$ and hence  $\wt{\cP}_A$  is a local isomorphism $(Z_A,A)\to(\Omega_{\Lambda},\cP(A))$ (injectivity and surjectivity of the local period map for compact hyperk\"ahler manifolds). Now suppose that $[A]\in \gT$. By hypothesis $\delta([A])\notin\gT$, and since $\delta(\gN)=\gN$ it follows that $\delta([A])\in(\gM\setminus\gN\setminus\gT)$; hence $\gp$ is open at $[A]$ by openness at $\delta([A])$ (which we have just proved) and~\eqref{idiota}. 
\end{proof}
Next we notice that $\DD_{\Lambda}$ is $\QQ$-factorial. In fact by Lemma~7.2 in Ch.~7 of~\cite{sata} there exists a torsion-free subgroup $G<\wt{O}(\Lambda)$ of finite index. Thus the natural map $\pi\colon (G\backslash \Omega_{\Lambda})\to \DD_{\Lambda}$ is a finite map of quasi-projective varieties, and $(G\backslash \Omega_{\Lambda})$ is smooth: it follows that  $\DD_{\Lambda}$ is $\QQ$-factorial. (If $D$ is a  divisor of $\DD_{\Lambda}$ then $(\deg\pi)D=\pi_{*}(\pi^{*}D)$ is the push-forward of a Cartier divisor, hence Cartier.) Since $\DD_{\Lambda}$ is $\QQ$-factorial and the restriction of $\gp$ to $(\gM\setminus\gN)$ is a birational map 
\begin{equation}\label{restper}
(\gM\setminus\gN)\to\DD_{\Lambda}
\end{equation}
the exceptional set of~\eqref{restper} is either empty or of pure codimension $1$. By~\Ref{clm}{granaper} there are no components of codimension $1$ in the  the exceptional set of~\eqref{restper}, hence it is empty. This proves that~\eqref{restper} is an open embedding.

\end{document}